\documentclass[11pt,leqno]{amsart}
\usepackage{amsthm,amsfonts,amssymb,amsmath,oldgerm}
\numberwithin{equation}{section}
\usepackage{fullpage}
\usepackage{mathtools}
\usepackage{color}
\usepackage{amsfonts}
\usepackage{calrsfs}
\usepackage{graphicx}
\DeclareMathAlphabet{\pazocal}{OMS}{zplm}{m}{n}

\def\eps{\varepsilon }

\newcommand\R{\mathbb R}

\def\eps{\varepsilon}

\newcommand\br{\begin{remark}}
\newcommand\er{\end{remark}}
\newcommand\bp{\begin{pmatrix}}
\newcommand\ep{\end{pmatrix}}
\newcommand{\be}{\begin{equation}}
\newcommand{\ee}{\end{equation}}
\newcommand\ba{\begin{equation}\begin{aligned}}
\newcommand\ea{\end{aligned}\end{equation}}

\newcommand{\bap}{\begin{app}}
\newcommand{\eap}{\end{app}}
\newcommand{\begs}{\begin{exams}}
\newcommand{\eegs}{\end{exams}}
\newcommand{\beg}{\begin{example}}
\newcommand{\eeg}{\end{exaplem}}
\newcommand{\bpr}{\begin{proposition}}
\newcommand{\epr}{\end{proposition}}
\newcommand{\bt}{\begin{theorem}}
\newcommand{\et}{\end{theorem}}
\newcommand{\bc}{\begin{corollary}}
\newcommand{\ec}{\end{corollary}}
\newcommand{\bl}{\begin{lemma}}
\newcommand{\el}{\end{lemma}}
\newcommand{\bd}{\begin{definition}}
\newcommand{\ed}{\end{definition}}
\newcommand{\brs}{\begin{remarks}}
\newcommand{\ers}{\end{remarks}}

\newcommand{\A }{\mathcal{A}}

\newcommand{\RR}{{\mathbb R}}

\newcommand{\CC}{{\mathbb C}}

\newcommand{\ta}{{\widetilde{a}}}
\newcommand{\tA}{{\widetilde{A}}}
\newcommand{\tB}{{\widetilde{B}}}
\newcommand{\tF}{{\widetilde{F}}}
\newcommand{\tM}{{\widetilde{M}}}
\newcommand{\tN}{{\widetilde{N}}}
\newcommand{\tS}{{\widetilde{S}}}
\newcommand{\tU}{{\widetilde{U}}}
\newcommand{\tW}{{\widetilde{W}}}
\newcommand{\tte}{{\widetilde{\theta}}}
\newcommand{\rmi}{{\mathrm{i}}}
\newcommand{\rmd}{{\mathrm{d}}}
\newcommand{\rms}{{\mathrm{s}}}

\newcommand{\rmc}{{\mathrm{c}}}

\newcommand{\rmb}{{\mathrm{b}}}
\newcommand{\ooa}{{\overline{a}}}

\newcommand{\im}{{\rm im }}

\newtheorem{theorem}{Theorem}[section]
\newtheorem{proposition}[theorem]{Proposition}
\newtheorem{corollary}[theorem]{Corollary}
\newtheorem{lemma}[theorem]{Lemma}

\theoremstyle{remark}
\newtheorem{remark}[theorem]{Remark}
\theoremstyle{definition}
\newtheorem{definition}[theorem]{Definition}

\newtheorem{example}[theorem]{Example}

\newcommand\cA{{\mathcal A}}
\newcommand\cB{{\mathcal B}}

\newcommand\cV{{\mathcal V}}

\newcommand\cL{{\mathcal L}}

\newcommand\cE{{\mathcal E}}
\newcommand\cF{{\mathcal F}}

\newcommand\cM{{\mathcal M}}

\newcommand\bH{{\mathbb H}}

\newcommand\bW{{\mathbb W}}

\newcommand\bX{{\mathbb X}}

\newcommand{\tH}{{\widetilde{H}}}

\newcommand{\tT}{{\widetilde{T}}}
\newcommand{\tC}{{\widetilde{C}}}

\newcommand{\tlambda}{{\widetilde{\lambda}}}

\newcommand{\ohh}{\overline{h}}
\newcommand{\og}{\overline{g}}
\newcommand{\ow}{\overline{w}}

\newcommand{\dom}{\text{\rm{dom}}}

\newcommand{\hatt}{\widehat}
\newcommand{\beq}{\begin{equation}}
\newcommand{\eeq}{\end{equation}}

\allowdisplaybreaks
\title{Uniform bounds of families of analytic semigroups and Lyapunov Linear Stability of planar fronts
}

\author{ Yuri Latushkin}
\address{University of Missouri, Columbia, MO 65211}
\email{latushkiny@missouri.edu}
\thanks{\hspace*{-0.17in}\textbf{AMS MSC 2020 Mathematics Subject Classification:} 37L15, 47D06, 34G10\\\textbf{Keywords:} Analytic semigroups, uniform exponential bounds, planar traveling waves, Lyapunov linear stability, reaction-diffusion systems, bidomain equation. \\Y.L. was supported by the NSF grants  DMS-1710989 and DMS-2106157, and would like to thank the Courant Institute of Mathematical Sciences and especially Prof.\ Lai-Sang Young for the opportunity to visit CIMS.\\A. P. research  was partially supported under the Simons Foundation Grant nr. 524928.}
\author{Alin Pogan}
\address{Miami University, Oxford, OH 45056}
\email{pogana@miamioh.edu}
\begin{document}

\begin{abstract}
We study families of analytic semigroups, acting in a Banach space, and depending on a parameter,  and give sufficient conditions for existence of uniform  with respect to the parameter norm bounds using spectral properties of the respective semigroup generators. In particular,  we use estimates of the resolvent operators of the generators along vertical segments to estimate the growth/decay rate of the norm for the family of analytic semigroups. These results are applied to prove the Lyapunov linear stability of planar traveling waves of systems of reaction-diffusion equations,  and the bidomain equation,  important in electrophysiology.
\end{abstract}

\maketitle

\vspace{0.3cm}
\begin{minipage}[h]{0.48\textwidth}
	\begin{center}
		University of Missouri \\
		Department of Mathematics\\
		810 East Rollins Street\\ Columbia, MO 65211, USA
	\end{center}
\end{minipage}
\begin{minipage}[h]{0.48\textwidth}
\begin{center}
Miami University\\
Department of Mathematics\\
301 S. Patterson Ave.\\
Oxford, OH 45056, USA
\end{center}
\end{minipage}

\vspace{0.3cm}


\section{Introduction }\label{s1}

The problem of finding optimal bounds for the norm of a semigroup of linear operators is very important and well-studied in the asymptotic theory of 
semigroups, \cite{ABHN,CL,DK,EN,lunardi,P}. Of particular importance are characterizations that relate the growth/decay bounds to the spectral properties of the generator of the semigroup. We mention here the celebrated Gearhard-Pr\" uss Theorem, \cite{Ge,Pr}, and refer to \cite[Section 5.7]{ABHN} for further references  or to \cite{HS1,HS2,LY} for more recent results. 

For families of semigroups depending on a parameter, the problem of finding estimates for their norms that are {\em uniform} with respect to the parameter is significantly less studied.
In this paper we aim to find sufficient conditions that guarantee uniform with respect to the parameter exponential decay estimates for families of analytic semigroups in Banach spaces.
These results are important in the study of Lyapunov linearized stability of traveling waves in systems of partial differential equations. In many instances, such as parabolic systems of partial differential equations, the linearization along the wave is a \textit{sectorial} operator, hence it generates an \textit{analytic semigroup}. In the case when the linearization along a planar wave is a differential operator in a multidimensional domain, the generator of the semigroup is typically similar to a multiplication operator by an operator valued function of certain parameter, the dual variable. The Lyapunov linearized stability problem is thus equivalent to the existence of \textit{uniform} with respect to the parameter bounds for a family of analytic semigroups. We present two specific cases, arising in the theory of reaction-diffusion systems and the bidomain equation, illustrating these ideas in Section~\ref{s4} and Section~\ref{s5}. 

It is well-known that for any strongly continuous semigroup of linear operators $\{T(t)\}_{t\geq 0}$ on a Banach space $\bX$ there exist two constants $L\geq1$ and $\omega\in\RR$ such that 
\begin{equation}\label{semigroup-exp-growth}
\|T(t)\|\leq Le^{\omega t}\;\mbox{for any}\;t\geq0.
\end{equation}
We recall that the infimum (which might not be the minimum) of all $\omega$ for which there exists a constant $L\geq 1$ such that (1.1) holds is called the semigroup {\em growth bound}, and  is denoted $\omega_0(T)$. Moreover, in the case of an analytic semigroup one has $\omega_0(T)=\rms(A):=\sup\mathrm{Re}\,\sigma(A)$, where $\rms(A)$ is the {\em spectral bound} and $\sigma(A)$ is  the spectrum of the generator $A$ of the semigroup $\{T(t)\}_{t\geq 0}$. We stress that it is much harder to find a direct formula, or even an estimate, for the constant $L$ in \eqref{semigroup-exp-growth}, unless one imposes additional conditions on the analytic semigroup $\{T(t)\}_{t\geq 0}$ or its generator. A classical example of such condition reads as follows: If $A$ is the generator of an analytic semigroup $\{T(t)\}_{t\geq 0}$ and $A-\omega I_\bX$ is dissipative, then \eqref{semigroup-exp-growth} holds for $L=1$, cf., e.g.,  \cite[Proposition 3.7.16]{ABHN}. 

In the case of {\em families} of semigroups whose generators depend on a parameter $\alpha$, real or complex, the constants $L$ and $\omega$ from \eqref{semigroup-exp-growth} might depend on $\alpha$ as well. By studying the spectrum of the generator and the properties of its resolvent, one can find a convenient growth rate. However, even if we use more advanced results, such as the 
celebrated Gearhard-Pr\" uss theorem (\cite{Ge,Pr}) or later results of Helfer-Sj\"{o}strand (\cite{HS1,HS2}), the constant $L=L(\alpha)$ is quite often such that $\sup_{\alpha} L(\alpha)=\infty$. 

Usually,  one aims at finding the best possible decay rate or the smallest growth rate, that is, the smallest $\omega$ in \eqref{semigroup-exp-growth}, may be at the expanse of making the constant $L$ large. However,  in this paper, studying exponential decay and stability of a family of semigroups depending on a parameter, we are not necessarily interested in the most optimal decay rate,  but rather just in the order of magnitude of the decay rate when the parameter changes. In this context we are willing to give up some decay to get a constant in front of the exponential term which is uniform with respect to the parameters in the system. In a sense, this philosophy is opposite to what has been used in many papers, for example, in  \cite{HS1,HS2,LY}. 

In the current paper we study uniform stability of a family of semigroups of operators, whose generators are bounded perturbations of a sectorial operator. We start with assumptions on the operator. 
 
\noindent{\bf Hypothesis (H1).}
We assume that $A:\rm{dom}(A)\subseteq\bX\to\bX$ is a \textit{sectorial} linear operator, that is, there exists $a\in\RR$, $\theta\in(\frac{\pi}{2},\pi)$ and $M_0>0$ such that
\begin{equation}\label{sectorial-B}
\Omega_{a,\theta}:=\{\lambda\in\CC:\lambda\ne a, |\mathrm{arg}(\lambda)-a|<\theta\}\subseteq\rho(A),\; \|R(\lambda,A)\|\leq \frac{M_0}{|\lambda-a|}\;\mbox{for any}\;\lambda\in\Omega_{a,\theta}.
\end{equation}
We recall that there are several concepts of sectorial operators in the literature. In the current paper we use the spectral conditions given above in \eqref{sectorial-B}. See Remark~\ref{r2.2} for a more detailed discussion.

It is well-known that if $A$ satisfies Hypothesis (H1) then it generates an analytic semigroup that can be evaluated using contour integration along a counterclockwise oriented path surrounding the spectrum of $A$. These formulas allow us to find estimates of the type \eqref{semigroup-exp-growth} and find explicit formulas not only for $\omega$, but also for $L$. 
We revisit several results known in the literature with explicit aim of controlling the constant $L$, which are very useful in the sequel, see Lemma~\ref{l2.4} and Remark~\ref{r2.5}. Next, we add to Hypothesis (H1) some more assumptions on $A$.

\noindent{\bf Hypothesis (H2).} We assume the following conditions on the spectrum of the operator $A$:
\begin{enumerate}
\item[(i)] $\sup\,\mathrm{Re}\big(\sigma(A)\setminus\{0\}\big)\leq-\nu$ for some $\nu>0$;
\item[(ii)] $0$ is a \textit{semisimple eigenvalue} of finite multiplicity of $A$.
\end{enumerate}
If the linear operator $A$ is obtained as the linearization of a PDE along a traveling wave, then Hypothesis (H2) is equivalent to what is sometimes called \textit{conditional exponential stability} of the wave. This situation is very common, see examples from Section~\ref{s4} and Section~\ref{s5}.

Next, we consider a family of operators $A_\alpha:=A+E(\alpha)$, $\alpha\geq0$, where the perturbation $E:[0,\infty)\to\mathcal{B}(\bX)$ is such that $E(0)=0$. We denote by $\{T_\alpha(t)\}_{t\geq 0}$ the semigroup of linear operators generated by $A_\alpha$, $\alpha\geq 0$. We assume the following conditions on the perturbed family of operators.

\noindent{\bf Hypothesis (H3).} The function $E:[0,\infty)\to\mathcal{B}(\bX)$ is continuous in the operator norm topology. Moreover, there exists  an increasing function $q:[0,\infty)\to[0,\infty)$ such that 
\begin{enumerate}
\item[(i)] There exists $M_1>0$ independent of $\alpha$ such that $\sigma(A_\alpha)\subset\{\lambda\in\CC:\mathrm{Re}\lambda\leq -q(\alpha)\}$ and
\begin{equation}\label{spectram-A-alpha}
\|R(\lambda,A_\alpha)\|\leq\frac{M_1}{\mathrm{Re}\lambda+q(\alpha)}\quad\mbox{whenever}\quad\mathrm{Re}\,\lambda>-q(\alpha),\;\alpha\geq 0;
\end{equation}
\item[(ii)] $q(\alpha)=q_1\alpha$ for any $\alpha\in [0,q_2]$, for some $q_1,q_2>0$;
 \item[(iii)]$\ell:=\limsup\limits_{\alpha\to\infty}\frac{\|E(\alpha)\|}{q(\alpha)}<\infty$.
\end{enumerate}
We note that the estimate \eqref{spectram-A-alpha} is necessary for the uniform in $\alpha$ exponential stability of the family of semigroups generated by $A_\alpha$, $\alpha\geq 0$. In Section~\ref{s4} and Section~\ref{s5} we present examples of families of operators that satisfy Hypothesis (H3). Moreover, a condition on the spectrum of the perturbed operator $A_\alpha$, $\alpha\geq0$, is needed to achieve uniform in $\alpha$ exponential stability of the semigroups $\{T_\alpha(t)\}_{t\geq 0}$, $\alpha\geq 0$. Indeed, in general the spectrum of the perturbed operator $A_\alpha$ might not be stable, even if the space $\bX$ is finite dimensional and $E(\alpha)=\alpha W_0$, $\alpha\geq 0$, where $W_0$ is a constant, bounded, self-adjoint, uniformly negative definite linear operator. See Appendix~\ref{s-a} for simple counterexamples.

Assuming Hypotheses (H1)-(H3) we show that the family of operators $A_\alpha$, $\alpha\geq 0$, is uniformly sectorial, that is $A_\alpha$ satisfies \eqref{sectorial-B} and all the relevant constants can be chosen independently on $\alpha$. Using this crucial result, we can prove that the family of semigroups generated by $A_\alpha$ is exponentially stable uniformly for  $\alpha\geq \delta$, for any $\delta>0$. To prove the uniform in $\alpha$ exponential estimate for the norm of the semigroup $\{T_\alpha(t)\}_{t\geq 0}$ for $\alpha$ in a neighborhood of $0$, we first find a spectral decomposition of the space $\bX$ into two subspaces invariant under $A_\alpha$. In this decomposition one spectral subspace is such that the spectrum of the restriction $A_\alpha$ is uniformly bounded away from the imaginary axis, and the other subspace is finite dimensional. To construct such a decomposition we use a transformation operator borrowed from
the classical work of Kato \cite[Chapter II, Section 4.2]{Kato} and Daletskii--Krein, \cite[Chapter 4, Section 1]{DK}. To ensure that the spectral and semigroup estimates are uniformly bounded in $\alpha$, we assume the following.

\noindent{\bf Hypothesis (H4).} The function $E_0:(0,\infty)\to\mathcal{B}(\bX)$ defined by $E_0(\alpha)=\frac{1}{\alpha} E_0(\alpha)$  has the property
\begin{equation}\label{representationE0}
\sup_{\alpha\in(0,1]}\|E_0(\alpha)\|<\infty.
\end{equation}	
This hypothesis guarantees that there exists a positive number $\eps_0>0$, \textit{independent of} $\alpha$, such that the circle of radius $\frac{\nu}{4}$ centered at the origin is contained in the resolvent set $\rho(A_\alpha)$, for any $\alpha\in [0,\eps_0]$, where $\nu>0$ was introduced in Hypothesis (H2) (see Lemma~\ref{new-l3.6} for more details). This result is essential in finding the desired spectral decomposition of the space $\bX$. Once this step is achieved, we use contour integration and the uniform sectorial property of the family of generators $A_\alpha$ to prove the uniform in $\alpha$ exponential stability of the restriction of the semigroup generated by $A_\alpha$, $\alpha\in [0,\eps_0]$, to the subspace where its spectrum is bounded away from the imaginary axis uniformly in $\alpha$.

Next, we turn our attention to the restriction of $A_\alpha$, $\alpha\in [0,\eps_0]$, to the finite dimensional subspace where all of its eigenvalues are of order $\mathcal{O}(\alpha)$. The most important step here is to find an expansion of this restriction near $\alpha=0$, which is so that all the eigenvalues of the leading order term are with negative real part. Moreover, the remainder of the leading order term is of order $o(\alpha)$ and bounded uniformly in $\alpha$ in a neighborhood of $0$. To be able to obtain such an expansion we need the following additional smoothness assumption on the perturbation.

\noindent{\bf Hypothesis (H5)}. 
There exists $r:[0,\infty)\to[0,\infty)$, a continuous, increasing function such that $r(0)=0$, $\lim_{\alpha\to\infty}r(\alpha)=\infty$, and the operator valued function $E_0$ defined in \eqref{representationE0} can be extended continuously at $0$ and satisfies the inequality
\begin{equation}\label{E0-bound}
\|E_0(\alpha)-E_0(0)\|\leq r(\alpha)\quad\mbox{for any}\quad\alpha\geq 0.
\end{equation}    
We note that Hypotheses (H4) and (H5) are automatically satisfied if the perturbation $E$ is analytic. The first main result of the paper is the following.
\begin{theorem}\label{t1.1}
Assume Hypotheses (H1)-(H5). Then, for an effectively computed in  equation \eqref{est-main} below function $M:(0,1)\to(0,\infty)$, independent of the perturbation variable $\alpha\geq0$, the following estimate holds,
\begin{equation}\label{t1.1.1}
\|T_\alpha(t)\|\leq M(\varkappa)e^{-\varkappa q(\alpha)t}\;\mbox{for any}\;t\geq 0,\alpha\geq 0, \varkappa\in(0,1).
\end{equation}
The function $M$ depends on the unperturbed operator $A$, $E_0(0)$ and the functions $\|E(\cdot)\|$, $q$, $r$ and related constants introduced in Hypotheses (H3)-(H5).
\end{theorem}
In part, assumption (H3)(i) yields that $\omega_0(T_\alpha)=\rms(A_\alpha)\leq -q(\alpha)$ as the semigroup $\{T_\alpha(t)\}_{t\geq 0}$ is analytic for any $\alpha\geq0$. Of course, we cannot replace the decay rate $-\varkappa q(\alpha)$, $\varkappa\in(0,1)$, in \eqref{t1.1.1} by $\omega_0(T_\alpha)$, and thus by $-q(\alpha)$, unless additional hypotheses are imposed. In many applications, however, we are mostly interested merely in the order of the decay rate, and \eqref{t1.1.1} captures this feature. 

Next, we turn our attention to the special case when $0$ is a \textit{simple} eigenvalue of the unperturbed operator $A$. The conclusion of Theorem ~\ref{t1.1} can be obtained in this case without assuming Hypothesis (H5). Arguing in the same way as in the general case, we first establish an exponential estimate similar to \eqref{t1.1.1} for $\alpha$ away from $0$. In addition, when $\alpha$ is close to $0$ we can decompose the space $\bX$ in the sum of two invariant for $A_\alpha$ subspaces: One is such that the spectrum of the restriction of $A_\alpha$ to the subspace is away from the imaginary axis, and another is a one dimensional subspace. A modification of Hypothesis (H2) is imposed next.

\noindent{\bf Hypothesis (H2').} We assume the following conditions on the spectrum of the operator $A$:
\begin{enumerate}
	\item[(i)] $\sup\,\mathrm{Re}\big(\sigma(A)\setminus\{0\}\big)\leq-\nu$ for some $\nu>0$;
	\item[(ii)] $0$ is a \textit{simple eigenvalue} of $A$.
\end{enumerate}
\begin{theorem}\label{t1.2}
Assume Hypotheses (H1) (H2'), (H3) and (H4). Then, for an effectively computed in  equation \eqref{est-main-bis} below function $N:(0,1)\to(0,\infty)$, independent of the perturbation variable $\alpha\geq0$, the following estimate holds,
\begin{equation}\label{t1.2.1}
\|T_\alpha(t)\|\leq N(\varkappa)e^{-\varkappa q(\alpha)t}\;\mbox{for any}\;t\geq 0,\alpha\geq 0, \varkappa\in(0,1).
\end{equation}
The function $N$ depends on the unperturbed operator $A$ and the functions $\|E(\cdot)\|$, $q$ and related constants introduced in Hypothesis (H3).
\end{theorem}

Next, we present applications of Theorem~\ref{t1.1} and Theorem~\ref{t1.2} to Lyapunov linear stability of planar traveling waves in reaction-diffusion systems and the bidomain equation. The two models are known to exhibit planar traveling waves. The linearization of each of them along the wave (in the moving frame variables) is similar via the Fourier Transform to a multiplication operator by an operator valued function of a certain parameter (the dual variable). Thus, the Lyapunov linear stability of a planar traveling wave can be obtained by studying the stability of a family of analytic semigroups which is uniform with respect to the parameter. 

The reaction-diffusion system 
\begin{equation}\label{RD-Sys}
u_t=D\Delta_xu+F(u),\; t\geq 0,\;x=(x_1,\dots,x_m)^{\mathrm{T}}\in\RR^m,\;u\in\RR^k,
\end{equation}
has exponentially localized planar traveling wave solutions, that is, solutions of the form $u(x,t)=\ohh(x_1-ct)$, under appropriate conditions on the matrix $D\in\CC^{k\times k}$ and the vector valued function $F:\RR^k\to\RR^k$. Considering the equation in the moving frame variable $y=x-ct\mathrm{\mathbf{e_1}}$, with $\mathrm{\mathbf{e_1}}=(1,0,\dots,0)^{\mathrm{T}}\in\RR^m$, the linearization along the planar wave $\ohh$ is given by $\cL=D\Delta_y+cI_k\partial_{y_1}+\cM_{F'(\ohh)}$, considered as a densely defined linear operator on $L^2(\RR^m,\CC^k)$. Here $\cM_{F'(\ohh)}$ denotes the operator of multiplication on $L^2(\RR^m,\CC^k)$ by the bounded, matrix valued function $F'(\ohh(y_1))$. Taking Fourier Transform in the variables $(y_2,\dots,y_m)\in\RR^{m-1}$, the linear operator $\cL$ is unitary equivalent to $\cM_{\hatt L}$ the operator of multiplication on $L^2\big(\RR^{m-1},L^2(\RR,\CC^k)\big)$ by the operator valued function     
$\hatt L:\RR^{m-1}\to\mathcal{B}\big(H^2(\RR,\CC^k),L^2(\RR,\CC^k)\big)$, defined by $\hatt L(\xi)=D\partial_{y_1}^2+cI_k\partial_{y_1}+\cM_{V(\cdot,\xi)}$, where $V(y_1,\xi)=F'(\ohh(y_1))-|\xi|^2D$. We refer to Section~\ref{s4} for a more detailed discussion on this topic. Next, we assume the following hypothesis.

\noindent{\bf Hypothesis (RD).} The spectrum of $\hatt L(0)=D\partial_{y_1}^2+cI_k\partial_{y_1}+\cM_{F'(\ohh(\cdot))}$, the linearization along the one-dimensional problem, satisfies the following conditions:
\begin{enumerate}
	\item[(i)] $\sup\,\mathrm{Re}\big(\sigma(\hatt L(0))\setminus\{0\}\big)\leq-\nu$ for some $\nu>0$;
	\item[(ii)] $0$ is a \textit{semisimple eigenvalue} of finite multiplicity of $\hatt L(0)$.
\end{enumerate}
We are now ready to formulate our Lyapunov stability result for the case of reaction-diffusion equations where the diffusion rates of various components of $u$ in \eqref{RD-Sys} are close to each other.
\begin{proposition}\label{p1.3}
Assume Hypothesis (RD) and that the diffusion matrix $D$ is sufficiently close to $dI_k$, for some $d>0$, in the sense described in equation \eqref{D-condition} below. Then the family of semigroups generated by $\hatt L(\xi)$ is stable uniformly  with respect to $\xi\in\RR^{m-1}$. In particular, the front $\ohh$ is Lyapunov linearly stable.	
\end{proposition}

Next, we briefly recall the bidomain model in electrophysiology studied by Matano and Mori, see, e.g., \cite{MaMo},
\begin{equation}\label{bidomain}
\left\{\begin{array}{ll}
u_t=\nabla_x\cdot(A_i\nabla_x u_i)+f(u),\\
\nabla_x\cdot(A_i\nabla_x u_i+A_e\nabla_x u_e)=0,\\
u=u_i-u_e,\end{array}\right.t\geq0,\; x\in\RR^2.
\end{equation} 
Here, the scalar functions $u_i$ and $u_e$ represent the intracellular and extracellular voltages, $A_i,A_e\in\RR^{2\times2}$ are symmetric, positive definite matrices. Typically one has 
\begin{equation}\label{def-A-i-e} 
A_i=\begin{bmatrix}
1+\nu_1+\nu_2&0\\
0&1+\nu_2-\nu_1\end{bmatrix},\quad A_e=\begin{bmatrix}
1-\nu_1-\nu_2&0\\
0&1-\nu_2+\nu_1\end{bmatrix},\;\mbox{with}\;|\nu_1\pm\nu_2|<1.
\end{equation}
The function $f$ is of class  $\mathcal{C}^3$ and
of bistable type, for example $0$ and $1$ are two stable zeros of $f$, and there exists a unique unstable zero of $f$ in the interval $(0,1)$. For a very detailed discussion regarding the applications and importance of the bidomain Allen-Cahn model to cardiac electrophysiology we refer to \cite{MaMo} and the references therein. 

The bidomain model \eqref{bidomain} has planar wave solutions of the form
$$(u,u_i,u_e)(x,t)=(\ow,\ow_i,\ow_e)(x_1\cos\gamma+x_2\sin\gamma-ct),\; x=(x_1,x_2)\in\RR^2, t\geq 0, c,\gamma\in\RR.$$ 
Next, we pass to the moving frame coordinate system $(y_1,y_2)\in\RR^2$, where the new axes are chosen such that the wave travels in the direction of $y_1$. Moreover, we note that last two equations of the system \eqref{bidomain} are linear. Linearizing the system along the traveling wave solution and eliminating the variables $u_i$ and $u_e$, we obtain that the linearization is given by     
$\cA=-\cL_\gamma +c\partial_{y_1}+\cM_{f'(\ow)}$. The linear operator $\cL_\gamma:H^2(\RR^2)\to L^2(\RR^2)$ is given as the Fourier multiplier $\cL_\gamma=\cF^{-1}\cM_{Q_\gamma}\cF$. The function $Q_\gamma$ is a rational function whose coefficients depend on $\nu_1$, $\nu_2$ and $\gamma$ only. Moreover, it can be represented as follows:
\begin{equation}\label{rep-Q-gamma}
Q_\gamma(\xi_1,\xi_2)=\left\{\begin{array}{ll} \xi_2^2\Big(p\Big(\frac{\xi_1}{\xi_2}\Big)+g\Big(\frac{\xi_1}{\xi_2}\Big)\Big),& \xi_1\in\RR,\,\xi_2\in\RR\setminus\{0\},\\
N_0^2\xi_1^2,& \xi_1\in\RR,\,\xi_2=0, \end{array}\right.\;\mbox{where},
\end{equation}
\begin{equation}\label{def-p-g}
p(s)=N_0^2(s-\eta_1)^2+\eta_0,\quad g(s)=\frac{\beta_1s+\beta_0}{s^2+1}.
\end{equation}
The constants $N_0$, $\beta_0$, $\beta_1$, $\eta_0$ and $\eta_1$ depend on $\nu_1$ and $\nu_2$ and $\gamma$ only. Taking Fourier Transform $\cF_2$ with respect to $y_2\in\RR$, we note that the linear operator $\cA$ is unitary equivalent to $\cM_{\hatt A}$ the operator of multiplication on $L^2\big(\RR,L^2(\RR)\big)$ by the operator valued function     
$\hatt A:\RR\to\mathcal{B}\big(H^2(\RR),L^2(\RR)\big)$, $\hatt A(\xi_2)=-\cF_1^{-1}\cM_{Q_\gamma(\cdot,\xi_2)}\cF_1+c\partial_{y_1}+\cM_{f'(\ow)}$. 
Here $\cF_1$ denotes the Fourier transform with respect to the variable $y_1\in\RR$.
For more details we refer to Section~\ref{s5} and \cite{MaMo}.

From \eqref{rep-Q-gamma} and \eqref{def-p-g} it follows that $\hatt A(0)$ is a second order differential operator and there exists $M_\rmb>0$ such that
\begin{equation}\label{hat-A-0}
\big\|R\big(\lambda,\hatt A(0)\big)\big\|\leq\frac{M_\rmb}{|\lambda|}\;\mbox{whenever}\;\mathrm{Re}\,\lambda\geq 0, \lambda\ne 0.
\end{equation}	
We recall the following result from \cite[Corollary 3.3]{MaMo}:
\begin{equation}\label{spectral-stable}
\ow \;\,\mbox{is spectrally stable provided that}\;\, \eta_0>M_\rmb g_\Delta-\og.
\end{equation}
Here $g_{\mathrm{inf}}=\inf_{s\in\RR}g(s)$, $g_{\mathrm{sup}}=\sup_{s\in\RR}g(s)$, $\og=\frac{g_{\mathrm{inf}}+g_{\mathrm{sup}}}{2}$, $g_\Delta=\frac{g_{\mathrm{sup}}-g_{\mathrm{inf}}}{2}$, while $\eta_0$ was introduced in \eqref{def-p-g} and $M_\rmb$ was introduced in \eqref{hat-A-0}. In \cite{MaMo} the authors show that \eqref{spectral-stable} is met for certain values of the parameters. In this paper we prove that the sufficient condition of \eqref{spectral-stable} guarantees that the planar front $\ow$ is Lyapunov linearly stable. 
\begin{proposition}\label{p1.4}
Assume that $\eta_0>M_\rmb g_\Delta-\og$. Then,  the family of semigroups generated by $\hatt A(\xi_2)$ is stable uniformly with respect to $\xi_2\in\RR$. In particular, the front $\ow$ is Lyapunov linearly stable.	
\end{proposition}

\noindent{\bf Plan of the paper.} The paper is organized as follows. In Section ~\ref{s2} we discuss the two most common concepts of sectorial operators, and analyze several exponential bounds for analytic semigroups. In Section~\ref{s3} we prove Theorem~\ref{t1.1} and Theorem~\ref{t1.2}. Proposition~\ref{p1.3} and Proposition~\ref{p1.4} are proved in Section ~\ref{s4} and Section ~\ref{s5}, respectively.

\noindent{\bf A glossary of notation.} 
$L^p(\RR^m,\bX)$, $p\geq 1$, denotes the usual Lebesgue space on $\RR^m$ with values in a Banach space $\bX$, associated with the Lebesgue measure $\rmd x$ on $\RR^m$.
$H^s(\RR^m,\bX)$, $s> 0$, is the usual Sobolev space of $\bX$ valued functions. The open disc in $\CC$ centered at $a$ of radius $\eps>0$ is denoted by $D(a,\eps)$.
The identity operator on a Banach space $\bX$ is denoted by $I_\bX$. The set of bounded linear operators from a Banach space $\bX$ to itself is denoted by $\cB(\bX)$.
For an operator $B$ on a Banach space $\bX$, we use  $\dom(B)$, $\ker B$, $\im B$, $\sigma(B)$, and $B_{|\bW}$ to denote the domain, kernel, range, spectrum, adjoint and the restriction of $B$ to a subspace $\bW$ of $\bX$. In the case when the space is a Hilbert space $B^*$ denotes the adjoint operator. We denote by $\sigma_{\mathrm{disc}}(B)$ the set of isolated eigenvalues of finite algebraic multiplicity of the linear operator $B$, and by $\sigma_{\mathrm{ess}}(B)$ its complement in the spectrum of $B$. The direct sum of two subspaces $\bW_1$ and $\bW_2$ is denoted by $\bW_1\oplus \bW_2$. The operator of multiplication by a function $g$ is denoted by $\cM_g$. We use $\omega_0(T)$ or $\omega_0(A)$ to denote the growth bound of a semigroup $\{T(t)\}_{t\geq 0}$ with generator $A$. The spectral bound of the generator $A$ is defined by $\rms(A)=\sup\mathrm{Re}\,\sigma(A)$.

\section{Norm estimates of semigroups generated by sectorial operators}\label{s2}
In this section we assume that $\bX$ is a Banach space and $A:\rm{dom}(A)\subseteq\bX\to\bX$ is a sectorial operator generating an analytic semigroup of linear operators denoted $\{T(t)\}_{t\geq 0}$. There are various concepts of sectorial operators relevant to our setup that we are going to briefly discuss below.

First, we recall that for any $a\in\RR$ and $\theta\in (\frac{\pi}{2},\pi)$ we defined the sector with vertex at $a$ of angle $\theta$ the set
\begin{equation}\label{def-sector}
\Omega_{a,\theta}=\{\lambda\in\CC:\lambda\ne a, |\mathrm{arg}(\lambda)-a|<\theta\}.
\end{equation}
One can readily check that
\begin{equation}\label{def-sector2}
\CC\setminus\Omega_{a,\theta}=\{\lambda\in\CC:\mathrm{Re}\lambda\leq a,|\mathrm{Im}\lambda|\leq(\mathrm{Re}\lambda-a)\tan\theta\}.
\end{equation}
In the literature on semigroup of linear  operators (see, e.g., \cite{EN,lunardi,P}) the definition of a sectorial operator is given using its spectral properties, see \eqref{sectorial-B}. In the case of linear operators on a Hilbert space, in particular differential operators, several classical texts (\cite{S}) define sectorial operators using the numerical range. We recall that if $\bH$ is a Hilbert space and $B:\dom(B)\subseteq\bH\to\bH$ is a closed, densely defined linear operator, the numerical range of $B$ is defined by
\begin{equation}\label{numerical-range}
W(B)=\{\langle Bh,h\rangle: h\in\dom(B),\|h\|=1\}.
\end{equation}
\begin{definition}\label{d2.1}The operator $B:\dom(B)\subseteq\bH\to\bH$ is said to be \textit{numerical range sectorial} if there exists $a\in\RR$, $\theta\in(\frac{\pi}{2},\pi)$ such that $W(B)\subseteq\CC\setminus\Omega_{a,\theta}$.
\end{definition}	
\begin{remark}\label{r2.2} From \cite[Theorem 1.4]{A} one can readily check that any \textit{numerical range sectorial} operator is \textit{sectorial}, but not vice versa. Indeed, if $W(B)\subseteq\CC\setminus\Omega_{a,\theta}$ for some $a\in\RR$ and $\theta\in(\frac{\pi}{2},\pi)$, since $\Omega_{a,\theta}$ is open we have $\sigma(B)\subseteq\overline{W(B)}\subseteq \CC\setminus\Omega_{a,\theta}$. In addition,
\begin{equation}\label{r2.2.1}
\|R(\lambda,B)\|\leq \frac{1}{\mathrm{dist}(\lambda,S_{a,\theta})}=\frac{1}{\mathrm{dist}(\lambda-a,S_{0,\theta})}.
\end{equation}
Another simple computation shows that
\begin{equation}\label{r2.2.2}
{\rm dist}(z,S_{0,\theta})=\left\{\begin{array}{ll} |z|,& |{\rm arg}z|<\theta-\frac{\pi}{2},\\
|z|\sin{(\theta-|{\rm arg}z|)},& \theta-\frac{\pi}{2}\leq|{\rm arg}z|<\theta, \end{array}\right.
\end{equation}
for any $z\in\Omega_{0,\theta}$. From \eqref{r2.2.1} and \eqref{r2.2.2} we infer that
\begin{equation}\label{r2.2.3}
\|R(\lambda,B)\|\leq\frac{\csc{(\theta-\varphi)}}{|\lambda-a|}\;\mbox{for any}\;\lambda\in\Omega_{a,\varphi}\;\mbox{and any}\;\varphi\in (\frac{\pi}{2},\theta),
\end{equation}
and so $B$ is sectorial. 
\end{remark}
It is well-known that if the linear operator $A:\dom(A)\subseteq\bX\to\bX$ satisfies Hypothesis (H1) then $\|T(t)\|\leq Le^{at}$ for any $t\geq 0$, for some $L>0$. Our first task is to revisit the proof of this result and for each $\varphi\in(\frac{\pi}{2},\theta)$ find a constant $L=L(\varphi)$, not necessarily optimal, satisfying the estimate that depends only on $\varphi$ and $M_0$.
\begin{lemma}\label{l2.3}
Assume Hypothesis (H1). Then,
\begin{equation}\label{vertex-estimate}
\|T(t)\|\leq \frac{M_0(e\varphi-\sec\varphi)}{\pi}e^{at}\;\mbox{for any}\;t\geq 0\;\mbox{and any}\;\varphi\in(\frac{\pi}{2},\theta).
\end{equation}	
\end{lemma}
\begin{proof}
We fix $\varphi\in(\frac{\pi}{2},\theta)$, $t>0$ and $\varepsilon>0$. We introduce the curves in the complex plane given by
\begin{align}\label{l2.3.1}
\Lambda^\pm_{a,\varphi,\varepsilon}&=\{\lambda\in\CC:\mathrm{arg}(\lambda-a)=\pm\varphi,|\lambda-a|\geq\eps\}=\{a+se^{\pm\rmi\varphi}:s\geq\eps\}\nonumber\\
\Lambda^\rmc_{a,\varphi,\varepsilon}&=\{\lambda\in\CC:|\mathrm{arg}(\lambda-a)|\leq\varphi,|\lambda-a|=\eps\}=\{a+\eps e^{\rmi\zeta}:-\varphi\leq\zeta\leq\varphi\}.
\end{align}	
The path $\Lambda_{a,\varphi,\eps}$ is defined as the union $\Lambda^-_{a,\varphi,\eps}\cup\Lambda^\rmc_{a,\varphi,\eps}\cup\Lambda^+_{a,\varphi,\eps}$ oriented counterclockwise. Since $\Lambda_{a,\varphi,\frac{\eps}{t}}\subseteq\Omega_{a,\theta}$ for any $t>0$ and $\eps>0$ we have  (see \cite[Chapter 2]{lunardi})
\begin{equation}\label{l2.3.2}
T(t)=\frac{1}{2\pi\rmi}\int_{\Lambda_{a,\varphi,\frac{\eps}{t}}} e^{\lambda t}R(\lambda,A)\rmd\lambda\;\mbox{for any}\; t>0, \eps>0.
\end{equation}
Next, we estimate the contour integrals above using \eqref{sectorial-B}. Changing variables we have 
\begin{align}\label{l2.3.3}
\int_{\Lambda^\pm_{a,\varphi,\frac{\eps}{t}}} e^{\lambda t}R(\lambda,A)\rmd\lambda&=\int_{\frac{\eps}{t}}^{\infty}e^{(a+se^{\pm\rmi\varphi})t}R(a+se^{\pm\rmi\varphi},A)\,e^{\pm\rmi\varphi}\rmd s\nonumber\\&=\frac{e^{at\pm\rmi\varphi}}{t}\int_{\eps}^{\infty} e^{\xi e^{\pm\rmi\varphi}}R\big(a+\frac{\xi}{t}e^{\pm\rmi\varphi},A\big)\,\rmd \xi.
\end{align}
Since $\Lambda^\pm_{a,\varphi,\frac{\eps}{t}}\subseteq\Omega_{a,\theta}$ for any $t>0$ and $\eps>0$, from \eqref{sectorial-B} and \eqref{l2.3.3} we obtain that
\begin{equation}\label{l2.3.4}
\Big\|\int_{\Lambda^\pm_{a,\varphi,\frac{\eps}{t}}} e^{\lambda t}R(\lambda,A)\rmd\lambda\Big\|\leq\frac{e^{at}}{t}\int_{\eps}^{\infty}e^{\xi\cos\varphi}\frac{M_0}{|\frac{\xi}{t}e^{\pm\rmi\varphi}|}\rmd\xi=M_0e^{at}\int_{\eps}^{\infty}\frac{e^{\xi\cos\varphi}}{\xi}\rmd\xi.
\end{equation}	
Similarly, since $\Lambda^\rmc_{a,\varphi,\frac{\eps}{t}}\subseteq\Omega_{a,\theta}$ for any $t>0$ and $\eps>0$, one can readily check that
\begin{equation}\label{l2.3.5}
\Big\|\int_{\Lambda^\rmc_{a,\varphi,\frac{\eps}{t}}} e^{\lambda t}R(\lambda,A)\rmd\lambda\Big\|=\Big\|\frac{\eps\rmi e^{at}}{t}\int_{-\varphi}^{\varphi}e^{\eps e^{\rmi\zeta}}R(a+\frac{\eps}{t}e^{\rmi\zeta},A)\rmd\zeta\Big\|\leq M_0e^{at}\int_{-\varphi}^{\varphi}e^{\eps \cos\zeta}\rmd\zeta.
\end{equation}	
From \eqref{l2.3.2}, \eqref{l2.3.4} and \eqref{l2.3.5} we conclude that
\begin{equation}\label{l2.3.6}
\|T(t)\|\leq M_0C_0(\varphi)e^{at}\;\mbox{for any}\; t\geq 0,\;\mbox{where}\; C_0(\varphi):=\frac{1}{\pi}\inf_{\eps>0}\Big(\int_{\eps}^{\infty}\frac{e^{\xi\cos\varphi}}{\xi}\rmd\xi+\frac{1}{2}\int_{-\varphi}^{\varphi}e^{\eps \cos\zeta}\rmd\zeta\Big).
\end{equation}
Finally, we note that
\begin{equation}\label{l2.3.7}
C_0(\varphi)\leq \frac{1}{\pi}\Big(\int_{1}^{\infty}\frac{e^{\xi\cos\varphi}}{\xi}\rmd\xi+\frac{1}{2}\int_{-\varphi}^{\varphi}e^{\cos\zeta}\rmd\zeta\Big)\leq\frac{1}{\pi}\Big(\int_{1}^{\infty}e^{\xi\cos\varphi}\rmd\xi+e\varphi\Big)=\frac{e\varphi-\sec\varphi}{\pi},
\end{equation}
proving the lemma.
\end{proof}	
Next, we discuss how to improve estimate \eqref{vertex-estimate}, provided that the sectorial semigroup generator $A$ is such that $\sup\mathrm{Re}\,\sigma(A)<a$, where $a$ is the vertex of the sector. Such an estimate is important since for many second order elliptic differential operators one can immediately prove they are sectorial using G\aa rding inequality, but typically the vertex is positive. To formulate our result we introduce the operator-valued function
\begin{equation}\label{def-V}
\cV_A:(0,\infty)\times(\mathrm{s}(A),\infty)\times(\frac{\pi}{2},\theta)\to\mathcal{B}(\bX),\;\cV_A(t,\mu,\varphi)=\int_{-(a-\mu)|\tan\varphi|}^{(a-\mu)|\tan\varphi|}e^{\rmi st}R(\mu+\rmi s,A)\rmd s.
\end{equation}
\begin{lemma}\label{l2.4}
Assume Hypothesis (H1) and that the vertex $a$ from \eqref{sectorial-B} is such that $a>\rms(A)=\sup\mathrm{Re}\,\sigma(A)$. Then,
\begin{equation}\label{V-estimate}
\|T(t)\|\leq \frac{M_0e^{\mu t}}{\pi(a-\mu)t}+\frac{e^{\mu t}}{2\pi}\sup_{t>0}\|\cV_A(t,\mu,\varphi)\|\quad\mbox{for any}\quad t>0, \mu\in\big(\mathrm{s}(A),a\big), \varphi\in(\frac{\pi}{2},\theta).
\end{equation}
\end{lemma}
\begin{proof} The proof uses the typical contour integral representation for sectorial operators. First we fix $\mu>s(A)$ and $\varphi\in(\frac{\pi}{2},\theta)$ and set $b=(a-\mu)|\tan\varphi|$ and $c=(a-\mu)|\sec\varphi|$. Next, we introduce the curves in the complex plane given by
\begin{equation}\label{l2.4.1}
\Gamma^\pm_{a,\mu,\varphi}=\{a+se^{\pm\rmi\varphi}:s\geq c\},\quad
\Gamma^\rmc_{a,\mu,\varphi}=\{\mu+\rmi s:-b\leq s\leq b\}.
\end{equation}
The path $\Gamma_{a,\varphi,\eps}$ is defined as the union $\Gamma^-_{a,\varphi,\eps}\cup\,\Gamma^\rmc_{a,\varphi,\eps}\cup\,\Gamma^+_{a,\varphi,\eps}$ oriented counterclockwise.
\begin{figure}[h]
\begin{center}
	\includegraphics[width=0.8\textwidth]{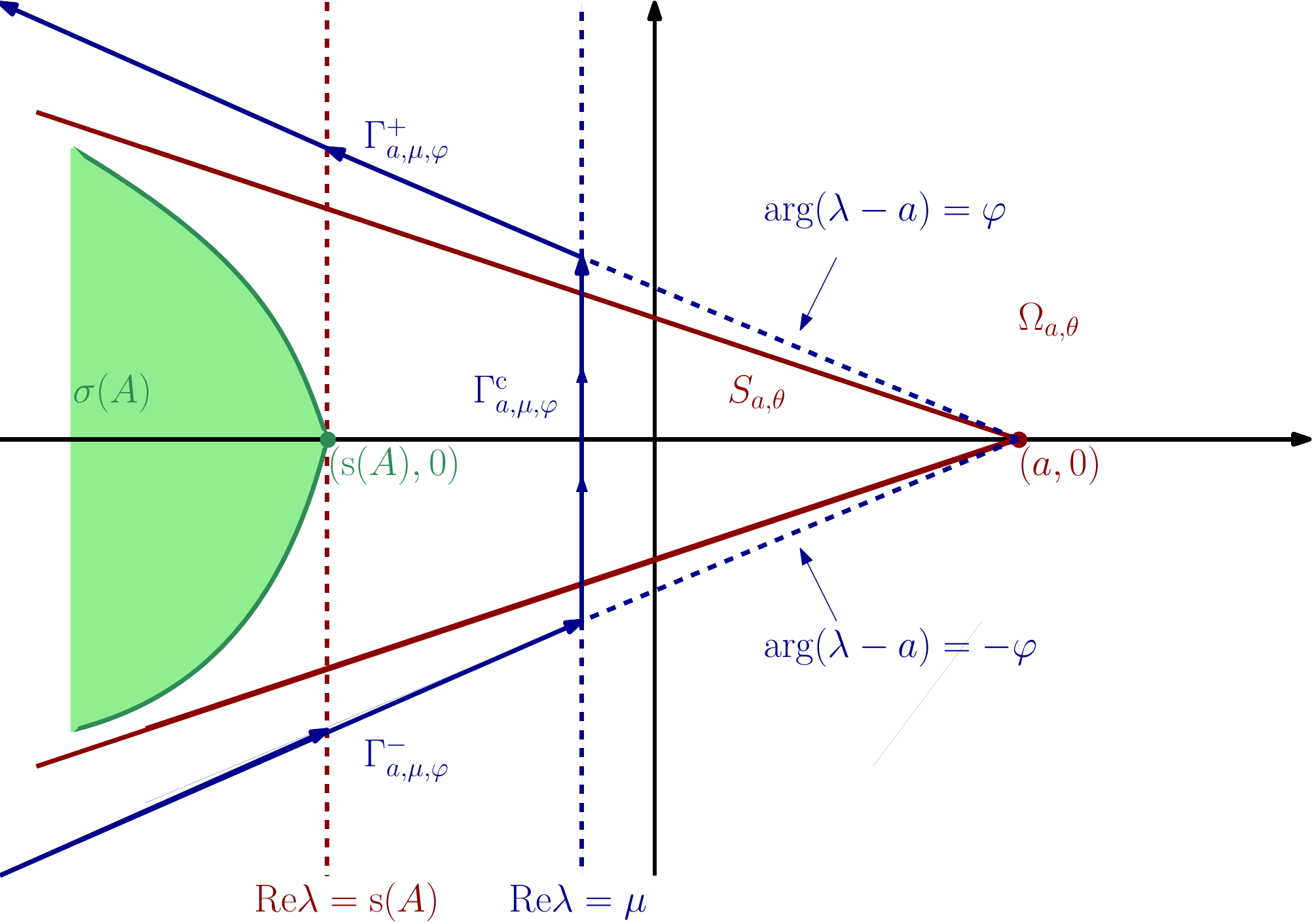}
	
	Figure 1. A plot of the spectrum of $A$, the sector $\Omega_{a,\theta}$ and the path $\Gamma_{a,\mu,\varphi}$ in the case when the vertex $a$ is positive.
\end{center}
\end{figure}
We note that this path is contained in the resolvent set of $A$ and surrounds, the spectrum of $A$	counterclockwise. It follows that (see, e.g., \cite[Chapter 2]{lunardi})
\begin{equation}\label{l2.4.2}
T(t)=\frac{1}{2\pi\rmi}\int_{\Gamma_{a,\mu,\varphi}} e^{\lambda t}R(\lambda,A)\rmd\lambda\;\mbox{for any}\; t>0.
\end{equation}
Since $\Gamma^\rmc_{a,\mu,\varphi}\subseteq\Omega_{a,\theta}$,  $a=\mu-c\cos\varphi$ from \eqref{sectorial-B} and \eqref{l2.4.1} we obtain that
\begin{align}\label{l2.4.3}
\Big\|\int_{\Gamma^\pm_{a,\mu,\varphi}} e^{\lambda t}R(\lambda,A)\rmd\lambda\Big\|&=\Big\|\int_{c}^{\infty}e^{(a+se^{\pm\rmi\varphi})t}R(a+se^{\pm\rmi\varphi},A)\,e^{\pm\rmi\varphi}\rmd s\Big\|\leq \int_{c}^{\infty}e^{(a+s\cos\varphi)t}\,\frac{M_0}{s}\rmd s\nonumber\\
&=\int_{c}^{\infty}e^{(\mu+(s-c)\cos\varphi)t}\,\frac{M_0}{s}\rmd s=M_0e^{\mu t}\int_{c}^{\infty}\frac{e^{(s-c)\cos\varphi t}}{s}\rmd s\nonumber\\
&=M_0e^{\mu t}\int_{0}^{\infty}\frac{e^{-\xi t}|\sec\varphi|}{c+|\sec\varphi|\xi}\rmd \xi=M_0e^{\mu t}\int_{0}^{\infty}\frac{e^{-\xi t}}{a-\mu+\xi}\rmd \xi\nonumber\\&\leq M_0e^{\mu t}\int_{0}^{\infty}\frac{e^{-\xi t}}{a-\mu}\rmd \xi=\frac{M_0e^{\mu t}}{(a-\mu)t}.
\end{align}
Unlike $\Lambda^\rmc_{a,\varphi,\eps}$ used in Lemma~\ref{l2.3} above, $\Gamma^\rmc_{a,\mu,\varphi}$ is not contained in the sector $\Omega_{a,\theta}$. To finish the proof of the lemma, we note that
\begin{equation}\label{l2.4.4}
\Big\|\int_{\Gamma^\rmc_{a,\mu,\varphi}} e^{\lambda t}R(\lambda,A)\rmd\lambda\Big\|=\Big\|\rmi e^{\mu t}\int_{-b}^{b} e^{\rmi st }R(\mu+\rmi s,A)\rmd s\Big\|\leq e^{\mu t}\sup_{t>0}\|\cV_A(t,\mu,\varphi)\|.
\end{equation}
The estimate \eqref{V-estimate} follows shortly from \eqref{l2.4.2}, \eqref{l2.4.3} and \eqref{l2.4.4}.
\end{proof}
We note that surprisingly the estimate \eqref{V-estimate} is more useful in the case when $a>0$. The main reason is that if $a>0$ and $\mu=-\varkappa q(\alpha)$, with the function $q$ as in Hypothesis (H3) and $\varkappa\in(0,1)$, then the constant $\frac{M_0}{\pi(a-\mu)}$ from \eqref{V-estimate} is uniformly bounded by $\frac{M_0}{\pi a}$. If $a=0$ and $\mu=-\varkappa q(\alpha)$ with $\alpha$ close to $0$, the constant $\frac{M_0}{\pi(a-\mu)}$ is of order $\mathcal{O}(\frac{1}{\alpha})$, making the estimate \eqref{V-estimate} not that useful.
\begin{remark}\label{r2.5}
Assuming Hypothesis (H1) and that $a>\rms(A)$, one can readily check that
\begin{equation}\label{l2.5.1}
\sup_{t>0}\|\cV_A(t,\mu,\varphi)\|\leq 2(a-\mu)|\tan\varphi|\sup\{\|R(\mu+\rmi s,A)\|:|s|\leq(a-\mu)|\tan\varphi|\}<\infty
\end{equation}
for any $\mu\in(s(A),a)$ and $\varphi\in(\frac{\pi}{2},\theta)$. From Lemma~\ref{l2.4} it follows that
\begin{equation}\label{l2.5.2}
\|T(t)\|\leq \frac{M_0e^{\mu t}}{\pi(a-\mu)t}+\frac{(a-\mu)|\tan\varphi|e^{\mu t}}{\pi}\sup\{\|R(\mu+\rmi s,A)\|:|s|\leq(a-\mu)|\tan\varphi|\}
\end{equation}
for any $t>0$ and any $\mu>\mathrm{s}(A)$, $\varphi\in(\frac{\pi}{2},\theta)$. This result for the special case of analytic semigroups on Banach spaces resembles the famous Gearhart-Pr\"uss result for semigroups on Hilbert spaces (\cite{Ge,Pr}), which was optimized by  Helffer and\ Sj\"{o}strand (\cite{HS1,HS2}). 
Indeed, the Gearhart-Pr\"uss theorem says that a strongly continuous semigroup is exponentially stable provided its generator has no spectrum in the right half plane and its resolvent operator is bounded along the whole imaginary axis.
In case of {\em analytic} semigroups, however, it is enough to estimate in \eqref{l2.5.2} the resolvent operator of $A$ along a {\em bounded} vertical segment, not the entire vertical line as it is needed for $C_0$-semigroups that are not necessarily analytic.
\end{remark}

Next, we collect a couple of other properties of the operator-valued function $\cV_A$. In particular, we are interested to find an alternative formula involving convolutions.
\begin{remark}\label{r2.6}
Setting $b=(a-\mu)|\tan\varphi|$ and using the fact that the resolvent operator of $A$ is the Laplace Transform of the analytic semigroup $\{T(t)\}_{t\geq 0}$, it follows that
\begin{align}\label{r2.6.1}
\cV_A(t,\mu,\varphi)&=\int_{-b}^{b}e^{\rmi st}\int_0^\infty e^{-(\mu+\rmi s)\xi}T(\xi)\rmd\xi\rmd s=\int_{-b}^{b}\int_0^\infty e^{\rmi(t-\xi)s}e^{-\mu\xi}T(\xi)\rmd\xi\rmd s\nonumber\\
&=\int_0^\infty \Big(\int_{-b}^{b}e^{\rmi(t-\xi)s}\rmd s\Big)e^{-\mu\xi}T(\xi)\rmd\xi=\int_0^\infty \frac{\sin{\big((a-\mu)(t-\xi)|\tan\varphi|\big)}}{t-\xi}e^{-\mu\xi}T(\xi)\rmd\xi
\end{align}
for any $\mu>s(A)$ and $\varphi\in(\frac{\pi}{2},\theta)$. This representation as a convolution shows that the estimate \eqref{V-estimate} is somewhat similar to convolution type the results of Latushkin and Yurov for the case of general semigroups on Banach spaces (\cite[Theorem 1.2]{LY}).
\end{remark}
\begin{remark}\label{r2.7}
From Hypothesis (H2)(ii) we infer that $0$ is a pole of order $1$ of $R(\cdot,A)$, hence $\lim_{\lambda\to0}{\lambda R(\lambda,A)}$ does exist in the operator norm. We infer that the operator valued function
\begin{equation}\label{r2.7.1}
\lambda\to\lambda R(\lambda,A):\Omega_{a,\theta}\cup\{\lambda\in\CC:\mathrm{Re}\lambda>-\nu\}\setminus\{0\}\to\mathcal{B}(\bX)
\end{equation}
can be extended analytically to $\Omega_{a,\theta}\cup\{\lambda\in\CC:\mathrm{Re}\lambda>-\nu\}$. For simplicity, we will use the same notation for this function and set $\big(\lambda R(\lambda,A)\big)_{|\lambda=0}=\lim_{\lambda\to0}{\lambda R(\lambda,A)}$.
\end{remark}
Next, we recall a sufficient condition that guaranties that a linear operator is sectorial due to A. Lunardi (\cite{lunardi}).
\begin{lemma}\label{lemma-Lunardi}
Assume that $A:\dom(A)\subseteq\bX\to\bX$ is a linear operator and $r\in\RR$ such that
\begin{equation}\label{lunardi1}
\{\lambda\in\CC:{\rm Re}\lambda\geq\omega\}\subseteq\rho(A)\quad\mbox{and}\quad\|\lambda R(\lambda,A)\|\leq N_0\quad\mbox{whenever}\quad{\rm Re}\lambda\geq\omega.
\end{equation}
Then, the operator $A$ is sectorial. More precisely,
\begin{equation}\label{lunardi2}
\Omega_{\omega,\tau}\subseteq\rho(A)\quad\mbox{and}\quad\|R(\lambda,A)\|\leq \frac{\sqrt{4N_0^2+1}}{|\lambda-\omega|}\quad\mbox{for any}\quad\lambda\in\Omega_{\omega,\tau},
\end{equation}
where $\tau=\pi-\arctan(2N_0)$. 	
\end{lemma}	
It is well-known that any bounded perturbation of a sectorial operator is also sectorial. To finish this section, we discuss a sufficient condition that guarantees
that the angle of the sector remains the same.
\begin{lemma}\label{l2.10}
Assume Hypothesis (H1) and $W\in\mathcal{B}(\bX)$. Then, \begin{equation}\label{perturbed-sectorial}
\Omega_{\ooa,\theta}\subseteq\rho(A+W),\; \|R(\lambda,A+W)\|\leq \frac{2M_0(1+\csc\theta)}{|\lambda-\ooa|}\;\mbox{for any}\;\lambda\in\Omega_{\ooa,\theta},
\end{equation}
where $\ooa=a+2M_0\|W\|\csc\theta$.
\end{lemma}
\begin{proof}
Fix $\lambda\in\Omega_{\ooa,\theta}$. Since $\ooa>a$, we have  $\lambda\in\Omega_{a,\theta}\subseteq\rho(A)$. Moreover, from Hypothesis (H1) we obtain that
\begin{equation}\label{l2.10.1}
|\lambda-a|\geq\mathrm{dist}(a,\Omega_{\ooa,\theta})=(\ooa-a)\sin\theta=2M\|W\|,\;\mbox{and thus}\;\|WR(\lambda,A)\|\leq\frac{M_0\|W\|}{|\lambda-a|}\leq\frac{1}{2}.
\end{equation}
It follows that $I_{\bX}-WR(\lambda,A)$ is invertible with bounded inverse and $\|\big(I_{\bX}-WR(\lambda,A)\big)^{-1}\|\leq 2$. Since $\lambda I_{\bX}-A-W=\big(I_{\bX}-WR(\lambda,A)\big)(\lambda I_{\bX}-A)$, from \eqref{l2.10.1} we infer that $\lambda\in\rho(A+W)$ and
\begin{align}\label{l2.10.2}
\|R(\lambda,A+W)\|&\leq\|R(\lambda,A)\|\|\big(I_{\bX}-WR(\lambda,A)\big)^{-1}\|\leq\frac{2M_0}{|\lambda-a|}=\frac{2M_0\big|\lambda-a-2M_0\|W\|\csc\theta\big|}{|\lambda-\ooa|\,|\lambda-a|}\nonumber\\&\leq\frac{2M_0}{|\lambda-\ooa|}\Big(1+\frac{2M_0\|W\|\csc\theta}{|\lambda-a|}\Big)\leq\frac{2M_0(1+\csc\theta)}{|\lambda-\ooa|},
\end{align}
proving the lemma.	
\end{proof}

\section{Norm estimates of families of analytic semigroups}\label{s3}
In this section we discuss the uniform exponential stability of families of analytic semigroups whose generators are bounded perturbations of a sectorial operator satisfying Hypotheses (H1) and (H2). Throughout this section we assume that $E:[0,\infty)\to\mathcal{B}(\bX)$ is such that $E(0)=0$, and the family of operators denotes by $A_\alpha:=A+E(\alpha)$, $\alpha\geq0$, satisfies Hypothesis (H3). In addition, we recall that $\{T_\alpha(t)\}_{t\geq 0}$ denotes the semigroup of linear operators generated by $A_\alpha$, $\alpha\geq 0$. 
\subsection{Basic Properties of families of analytic semigroups}\label{sec3-1}

First, we will show that the operator $A_\alpha$, $\alpha\geq 0$, is sectorial and all the constants from \eqref{sectorial-B} can be chosen independent of $\alpha$.
\begin{lemma}\label{l3.1}
Assume Hypotheses (H1)--(H3). Then, there exist $\ta>0$, $\tte\in(\frac{\pi}{2},\pi)$ and $\tM>0$, independent of $\alpha\geq0$, such that
\begin{equation}\label{uniform-sectorial-B}
\Omega_{\ta,\tte}\subseteq\rho(A_\alpha),\quad \|R(\lambda,A_\alpha)\|\leq \frac{\tM}{|\lambda-\ta|}\quad\mbox{for any}\quad\lambda\in\Omega_{\ta,\tte},\;\alpha\geq 0.
\end{equation}
The constants $\ta$, $\tte$ and $\tM$ depend on the unperturbed operator $A$ and the functions $\|E(\cdot)\|$ and $q$ and related constants in Hypotheses (H3).	
\end{lemma}
\begin{proof}
Without loss of generality we can assume that $a>0$, where $a$ is the vertex of the sector from Hypothesis (H1). From Hypothesis (H3)(iii) we have there exists $\alpha_0>0$ such that $\|E(\alpha)\|\leq (\ell+1)q(\alpha)$ for any $\alpha\geq \alpha_0$. From Hypotheses (H2) and (H3) respectively, it is clear that $\{\lambda\in\CC:\mathrm{Re}\lambda\geq a+1\}\subseteq\rho(A)\cap\rho(A_\alpha)$ for any $\alpha\geq \alpha_0$. Moreover, 
\begin{equation}\label{l3.1.1}
R(\lambda,A_\alpha)-R(\lambda,A)=R(\lambda,A_\alpha)E(\alpha)R(\lambda,A)\quad\mbox{whenever}\quad{\rm Re}\lambda\geq a+1,\;\alpha\geq \alpha_0.
\end{equation}
From \eqref{l3.1.1}, Hypothesis (H1) and Hypothesis (H3)(i) we obtain that
\begin{align}\label{l3.1.2}
\|\lambda R(\lambda,A_\alpha)\|&\leq\|\lambda R(\lambda,A)\|+\|R(\lambda,A_\alpha)\|\,\|E(\alpha)\|\,\|\lambda R(\lambda,A)\|\leq \frac{M_0|\lambda|}{|\lambda-a|}\Big(1+\frac{M_1\|E(\alpha)\|}{\mathrm{Re}\lambda+q(\alpha)}\Big)\nonumber\\
&\leq\frac{M_0|\lambda|}{|\lambda-a|}\Big(1+\frac{M_1\|E(\alpha)\|}{q(\alpha)}\Big)\leq\frac{|\lambda|}{|\lambda-a|}M_0\big(1+M_1(\ell+1)\big)\nonumber\\
&\leq M_0\big(1+M_1(\ell+1)\big)\Big(1+\frac{a}{|\lambda-a|}\Big)\leq M_0(1+a) \big(1+M_1(\ell+1)\big)
\end{align}
whenever ${\rm Re}\lambda\geq a+1$ and $\alpha\geq \alpha_0$. From Hypothesis (H3) we infer that $\sup_{\alpha\in[0,\alpha_0]}\|E(\alpha)\|<\infty$.  Lemma~\ref{l2.10} yields
\begin{equation}\label{l3.1.3}
\Omega_{a_1,\theta}\subseteq\rho(A_\alpha),\quad \|R(\lambda,A_\alpha)\|\leq \frac{2M_0(1+\csc\theta)}{|\lambda-a_1|}\quad\mbox{for any}\quad\lambda\in\Omega_{a_1,\theta},\;\alpha\in[0,\alpha_0],
\end{equation}
where $a_1=a+2M_0\sup_{\alpha\in[0,\alpha_0]}\|E(\alpha)\|\csc\theta\geq a$. It follows that $\{\lambda\in\CC:\mathrm{Re}\lambda\geq a_1+1\}\subseteq\rho(A_\alpha)$ for any $\alpha\in[0,\alpha_0]$ and
\begin{equation}\label{l3.1.4}
\|\lambda R(\lambda,A_\alpha)\|\leq \frac{2M_0(1+\csc\theta)|\lambda|}{|\lambda-a_1|}\leq 2M_0(1+\csc\theta)\Big(1+\frac{a_1}{|\lambda-a_1|}\Big)\leq 2M_0(1+\csc\theta)(1+a_1)
\end{equation}
whenever ${\rm Re}\lambda\geq a_1+1$ and $\alpha\in[0,\alpha_0]$. Since $a_1\geq a$ from \eqref{l3.1.2} and \eqref{l3.1.4} it follows that $\{\lambda\in\CC:\mathrm{Re}\lambda\geq a_1+1\}\subseteq\rho(A_\alpha)$ for any $\alpha\geq 0$ and
\begin{equation}\label{l3.1.5}
\|\lambda R(\lambda,A_\alpha)\|\leq M_0(1+a_1)\big(1+\max\{\csc\theta,M_1(\ell+1)\}\big)\quad\mbox{whenever}\quad{\rm Re}\lambda\geq a_1+1,\alpha\geq 0.
\end{equation}
Setting $\ta=a_1+1$, $\tN=M_0(1+a_1)(1+\max\{\csc\theta,M_1(\ell+1)\})$, $\tM=\sqrt{4\tN^2+1}$ and $\tte=\pi-\arctan(2\tN)$, from Lemma~\ref{lemma-Lunardi} and \eqref{l3.1.5} we infer that
\begin{equation}\label{l3.1.6}
\Omega_{\ta,\tte}=\Omega_{a_1+1,\tte}\subseteq\rho(A_\alpha)\;\mbox{and}\;\|R(\lambda,A_\alpha)\|\leq \frac{\sqrt{4\tN^2+1}}{|\lambda-a_1-1|}=\frac{\tM}{|\lambda-\ta|}\;\mbox{for any}\;\lambda\in\Omega_{\ta,\tte},\,\alpha\geq 0,
\end{equation}
proving the lemma.
\end{proof}
In the next lemma we apply Lemma~\ref{l2.4} to estimate the norm of the semigroup $\{T_\alpha(t)\}_{t\geq 0}$ for $\alpha$ away from $0$.
\begin{lemma}\label{l3.2}
Assume Hypotheses (H1)--(H3). Then the following estimate holds,
\begin{equation}\label{T-alpha-delta}
\|T_\alpha(t)\|\leq \overline{M}(\varkappa,\alpha)e^{-\varkappa q(\alpha)t}\quad\mbox{for any}\quad t\geq0,\;\alpha>0,\;\varkappa\in(0,1).
\end{equation}
Here the function $\overline{M}:(0,1)\times(0,\infty)\to(0,\infty)$ is defined by 
\begin{equation}\label{overline-M}\begin{split}
\overline{M}(\varkappa,\alpha)&=\frac{1}{\pi}\max\Bigg\{\frac{\tM q(\alpha)}{\ta+\varkappa q(\alpha)}+\frac{M_1\big(\ta+\varkappa q(\alpha)\big)|\tan(\frac{\tte}{2}+\frac{\pi}{4})|}{(1-\varkappa) q(\alpha)},\\
&\qquad\qquad\qquad\qquad\frac{\tM\Big(e(2\tte+\pi)-4\sec(\frac{\tte}{2}+\frac{\pi}{4})\Big)}{4}e^{\frac{\ta}{q(\alpha)}+1}\Bigg\}.\end{split}
\end{equation}
In particular, for any $\delta>0$ the family of semigroups $\{T_\alpha(t)\}_{t\geq0}$ is exponentially stable uniformly for $\alpha\geq\delta$.
\end{lemma}	
\begin{proof}
Fix $\varkappa\in (0,1)$, $\alpha>0$ and let $\mu=-\varkappa q(\alpha)$. We set $\varphi=\frac{\tte}{2}+\frac{\pi}{4}\in(\frac{\pi}{2},\tte)$ and $\widetilde{b}=(\ta-\mu)|\tan\varphi|$. From Hypothesis (H3) we have $\mu>\rm{s}(A_\alpha)$ and
\begin{equation}\label{l3.2.1}
\|R(\mu+\rmi s,A_\alpha)\|\leq\frac{M_1}{{\rm Re}(\mu+\rmi s)+q(\alpha)}=\frac{M_1}{\mu+q(\alpha)}=\frac{M_1}{(1-\varkappa)q(\alpha)}\quad\mbox{for any}\quad s\in\RR.
\end{equation}
Using definition \eqref{def-V} leads to
\begin{equation}\label{l3.2.2}
\|\cV_{A_\alpha}(t,\mu,\varphi)\|=\Big\|\int_{-\widetilde{b}}^{\widetilde{b}}e^{\rmi st}R(\mu+\rmi s,A_\alpha)\rmd s\Big\|\leq\frac{2M_1\widetilde{b}}{(1-\varkappa)q(\alpha)}=\frac{2M_1(\ta-\mu)|\tan(\frac{\tte}{2}+\frac{\pi}{4})|}{(1-\varkappa)q(\alpha)}
\end{equation}
for any $t>0$. From \eqref{uniform-sectorial-B}, \eqref{l3.2.2} and Lemma~\ref{l2.4} it follows that
\begin{align}\label{l3.2.3}
\|T_\alpha(t)\|&\leq\frac{\tM e^{\mu t}}{\pi(\ta-\mu)t}+\frac{e^{\mu t}}{2\pi}\cdot\frac{2M_1(\ta-\mu)|\tan(\frac{\tte}{2}+\frac{\pi}{4})|}{(1-\varkappa) q(\alpha)}\nonumber\\
&=\frac{e^{-\varkappa q(\alpha)t}}{\pi}\Bigg(\frac{\tM }{(\ta+\varkappa q(\alpha))t}+\frac{M_1(\ta+\varkappa q(\alpha))|\tan(\frac{\tte}{2}+\frac{\pi}{4})|}{(1-\varkappa)q(\alpha)}\Bigg)\;\mbox{for any}\;t>0.
\end{align}
If $t\geq\frac{1}{q(\alpha)}$ then $(\ta+\varkappa q(\alpha))t\geq\frac{\ta}{q(\alpha)}+\varkappa$, thus from \eqref{l3.2.3} we have 
\begin{equation}\label{l3.2.4}
\|T_\alpha(t)\|\leq\frac{e^{-\varkappa q(\alpha)t}}{\pi}\Big(\frac{\tM q(\alpha)}{\ta+\varkappa q(\alpha)}+\frac{M_1(\ta+\varkappa q(\alpha))|\tan(\frac{\tte}{2}+\frac{\pi}{4})|}{(1-\varkappa)q(\alpha)}\Big)\;\mbox{for any}\;t\geq\frac{1}{q(\alpha)}.
\end{equation}
From Lemma~\ref{l2.3} and Lemma~\ref{l3.1} we infer 
\begin{align}\label{l3.2.5}
\|T_\alpha(t)\|&\leq\frac{\tM(e\varphi-\sec\varphi)}{\pi}e^{\ta t}=\frac{\tM\Big(e(2\tte+\pi)-4\sec(\frac{\tte}{2}+\frac{\pi}{4})\Big)}{4\pi}e^{(\ta+\varkappa q(\alpha))t} e^{-\varkappa q(\alpha)t}\nonumber\\
&\leq \frac{\tM\Big(e(2\tte+\pi)-4\sec(\frac{\tte}{2}+\frac{\pi}{4})\Big)}{4\pi}e^{\frac{\ta}{q(\alpha)}+1} e^{-\varkappa q(\alpha)t}
\end{align}
for any $t\in\big[0,\frac{1}{q(\alpha)}\big)$. The estimate \eqref{T-alpha-delta} follows from \eqref{l3.2.4} and \eqref{l3.2.5}, with $\overline{M}(\varkappa,\alpha)$ given by \eqref{overline-M}. One can readily check that 
\begin{equation}\label{l3.2.6}
\overline{M}(\varkappa,\alpha)\leq \frac{1}{\pi}\max\Bigg\{\frac{\tM}{\varkappa}+\frac{M_1|\tan(\frac{\tte}{2}+\frac{\pi}{4})|\big(\frac{\ta}{q(\alpha)}+\varkappa\big)}{(1-\varkappa) },\frac{\tM\Big(e(2\tte+\pi)-4\sec(\frac{\tte}{2}+\frac{\pi}{4})\Big)}{4}e^{\frac{\ta}{q(\alpha)}+1}\Bigg\}.
\end{equation}
Since $\inf_{\alpha\geq\delta}q(\alpha)=q(\delta)>0$ for any $\delta>0$, by Hypothesis (H3), from \eqref{T-alpha-delta} and \eqref{l3.2.6} it follows that
for any $\delta>0$ the family of semigroups $\{T_\alpha(t)\}_{t\geq0}$ is exponentially stable uniformly for $\alpha\geq\delta$.
\end{proof}
\begin{remark}\label{r3.3}
In the case when the space $\bX$ is a \textit{Hilbert} space one can use the results of B. Helfer and J. Sj\"{o}strand (\cite{HS1,HS2}) to obtain an estimate similar to \eqref{T-alpha-delta}. Indeed, from Lemma~\ref{l2.3} and Lemma~\ref{l3.1} it follows that
\begin{equation}\label{r3.3.1}
\|T_\alpha(t)\|\leq \frac{\tM\Big(e(2\tte+\pi)-4\sec(\frac{\tte}{2}+\frac{\pi}{4})\Big)}{4\pi}e^{\ta t}\quad\mbox{for any}\quad t>0,\;\alpha>0.
\end{equation}
From Hypothesis (H3)(i) we have
\begin{equation}\label{r3.3.2}
\sup_{{\rm Re}\lambda\geq-\varkappa q(\alpha)}\|R(\lambda,A_\alpha)\|\leq\sup_{{\rm Re}\lambda\geq-\varkappa q(\alpha)}\frac{M_1}{\mathrm{Re}\lambda+q(\alpha)}\leq\frac{M_1}{(1-\varkappa)q(\alpha)}\quad\mbox{for any}\quad\alpha>0,\;\varkappa\in(0,1).
\end{equation}
From \cite[Proposition 2.1]{HS1} we obtain that
\begin{equation}\label{T-alpha-delta-Hilbert}
\|T_\alpha(t)\|\leq \overline{N}(\varkappa,\alpha)e^{-\varkappa q(\alpha)t}\quad\mbox{for any}\quad t>0,\;\alpha>0,
\end{equation}
where the function $\overline{N}:(0,1)\times(0,\infty)\to\RR$ is defined by
\begin{equation}\label{overline-N}
\overline{N}(\varkappa,\alpha)=\frac{\tM\Big(e(2\tte+\pi)-4\sec(\frac{\tte}{2}+\frac{\pi}{4})\Big)}{4\pi}\Bigg(1+\frac{\tM\Big(e(2\tte+\pi)-4\sec(\frac{\tte}{2}+\frac{\pi}{4})\Big)}{2\pi}\Big(1+\frac{M_1(\ta+\varkappa q(\alpha)}{(1-\varkappa)q(\alpha)}\Big)\Bigg).
\end{equation}
Here $\tM$, $\tte$ and $\ta$ are obtained in Lemma~\ref{l3.1} and $M_1$ is taken from Hypothesis (H3) (i). We note that $\overline{N}(\varkappa,\alpha)=\mathcal{O}(\frac{1}{q(\alpha)})$ and  $\overline{M}(\varkappa,\alpha)=\mathcal{O}(\frac{1}{q(\alpha)})$ as $\alpha\to0$, which shows that regardless of what approach is used one can conclude uniform exponential stability of the family of semigroups $\{T_\alpha(t)\}_{t\geq0}$ for $\alpha\geq\delta$, for any $\delta>0$, but not for $\alpha\geq0$. Also, from \eqref{overline-N} we note that $\sup_{\alpha\geq 0}\overline{N}(\varkappa,\alpha)<\infty$ only if $\ta=0$, which occurs only when the family of semigroups $\{T_\alpha(t)\}_{t\geq0}$, $\alpha\geq0$, is bounded uniformly in $\alpha\geq0$, a property which is as difficult to establish as the result of exponential decay of the family of semigroups with uniform in $\alpha$ constants. 
\end{remark}

\subsection{Norm Estimates of $\{T_\alpha(t)\}_{t\geq 0}$ when $\alpha$ is in a neighborhood of $0$.}\label{sec3-2}

To estimate the norm of the semigroup $\{T_\alpha(t)\}_{t\geq 0}$ when $\alpha$ is in a neighborhood of $0$, we need a different approach. The key part of the argument is to find a spectral decomposition of the Banach space $\bX$, into a sum of two closed subspaces invariant under the semigroup $\{T_\alpha(t)\}_{t\geq 0}$, such that the spectrum of the restriction of $A_\alpha$ to the two subspaces is either away from the imaginary axis or made up entirely of eigenvalues.
We first look for such a decomposition of the linear operator $A$. From Hypothesis (H2) we have the spectrum of $A$ has two disjoint parts, separated by the circle $\partial D(0,\frac{\nu}{2})$. We define $P_0\in\mathcal{B}(\bX)$ as the spectral projection relative to the spectral subset $\{0\}$ of $\sigma(A)$, given by the formula
\begin{equation}\label{def-P-0}
P_0=\frac{1}{2\pi\rmi}\int_{\partial D(0,\frac{\nu}{2})}R(\lambda,A)\rmd\lambda.
\end{equation}
\begin{remark}\label{r3.4}
We collect several well-known properties of the spectral projection $P_0$, see e.g., \cite{EN,lunardi,P}:
\begin{enumerate}
\item[(i)] $\mathrm{Im}P_0\subset\dom(A)$;
\item[(ii)] $\mathrm{Im}P_0$ and $\mathrm{Ker}P_0$ are invariant under $A$ and $T(t)$ for any $t\geq 0$;
\item[(iii)] The linear operator $\tA:=A_{|\dom(A)\cap\mathrm{Ker}P_0}$ is the generator of the semigroup $\{\tT(t)\}_{t\geq 0}=\{T(t)_{|\mathrm{Ker}P_0}\}_{t\geq 0}$ on $\mathrm{Ker}P_0$;
\item[(iv)] $\sigma(\tA)=\sigma(A)\setminus\{0\}$.
\end{enumerate}
\end{remark}	
In the next lemma we aim to refine the properties of $P_0$ by making use of Hypotheses (H1)--(H2).
\begin{lemma}\label{l3.5}
Assume Hypotheses (H1)--(H2). Then, the following assertions hold true
\begin{enumerate}
\item[(i)] $\mathrm{dim}\,\mathrm{Im}P_0<\infty$ and $A_{|\mathrm{Im}P_0}\equiv0$;
\item[(ii)] The operator $\tA$ is sectorial, and thus it generates an analytic semigroup. Moreover, there exists $M_2>0$ such that
\begin{equation}\label{est-sem-tilde-A}
\|\tT(t)\|\leq M_2e^{-\frac{7\nu}{8}t}\quad\mbox{for any}\quad t>0;\footnote{ In \eqref{est-sem-tilde-A} we can replace $-\frac{7\nu}{8}$ by $-\varkappa\nu$ for any $\varkappa\in(0,1)$, but not by $-\nu$. However, such a change would not enhance our main result, cf. Theorem~\ref{t1.1}, since for $\alpha$ close to $0$ under Hypothesis (H3) the decay rate $-\frac{7\nu}{8}$ is stronger than $-\varkappa q(\alpha)$, the decay rate in the main result.}
\end{equation}
\item[(iii)]The linear operators $A$ and $T(t)$, $t\geq0$, respectively, have the representations
\begin{equation}\label{representation-A-T}
A=\begin{bmatrix}
0&0\\
0&\tA\end{bmatrix},\quad T(t)=\begin{bmatrix}
I_{\mathrm{Im}P_0}&0\\
0&\tT(t)\end{bmatrix},\; t\geq 0,
\end{equation}
with respect to the decomposition $\bX=\mathrm{Im}P_0\oplus\mathrm{Ker}P_0$ (direct sum, not necessarily orthogonal).
\end{enumerate}	
\end{lemma}	
\begin{proof}
	(i) Since $0$ is a semisimple eigenvalue of finite multiplicity of $A$ we infer that $\mathrm{dim}\,\mathrm{Im}P_0<\infty$. Moreover, from Remark~\ref{r3.4}(i) and since $A$ is a closed linear operator we have 
	\begin{equation}\label{l3.5.1}
	AP_0=\frac{1}{2\pi\rmi}\int_{\partial D(0,\frac{\nu}{2})}AR(\lambda,A)\rmd\lambda=\frac{1}{2\pi\rmi}\int_{\partial D(0,\frac{\nu}{2})}\big(\lambda R(\lambda,A)-I_{\bX}\big)\rmd\lambda.
	\end{equation}
	By Remark~\ref{r2.7} we know that the function $\lambda\to\lambda R(\lambda,A):\Omega_{a,\theta}\cup\{\lambda\in\CC:\mathrm{Re}\lambda>-\nu\}\setminus\{0\}\to\mathcal{B}(\bX)$
	can be extended analytically to $\Omega_{a,\theta}\cup\{\lambda\in\CC:\mathrm{Re}\lambda>-\nu\}$, therefore from \eqref{l3.5.1} we obtain that $AP_0\equiv0$, hence $A_{|\mathrm{Im}P_0}\equiv0$.
	
\noindent (ii) From Remark~\ref{r3.4}(iv) and Hypothesis (H1) we have $\Omega_{a,\theta}\subset\rho(A)\subset\rho(\tA)$ and $R(\lambda,\tA)=R(\lambda,A)_{|\mathrm{Ker}P_0}$ for any $\lambda\in\Omega_{a,\theta}$, hence
\begin{equation}\label{l3.5.2}
\|R(\lambda,\tA)\|\leq\|R(\lambda,A)\|\leq\frac{M}{|\lambda-a|}\quad\mbox{for any}\quad\lambda\in\Omega_{a,\theta}.
\end{equation}
We conclude that the linear operator $\tA$ is sectorial, therefore the semigroup $\{\tT(t)\}_{t\geq 0}$ is analytic. From Hypothesis (H2)(i) it follows that
\begin{equation}\label{l3.5.3}
\omega_0(\tA)=\sup\mathrm{Re}\sigma(\tA)=\sup\mathrm{Re}\big(\sigma(A)\setminus\{0\}\big)\leq-\nu<-\frac{7\nu}{8}.
\end{equation}
From the definition of the growth rate of a semigroup one immediately concludes that
\begin{equation}\label{l3.5.4}
M_2=\sup_{t\geq 0}\big(e^{\frac{7\nu}{8}t}\|\tT(t)\|\big)<\infty,
\end{equation}
proving (ii). Assertion (iii) follows shortly from (i) and Remark~\ref{r3.4}(ii) and (iii).
\end{proof}
To construct a representation of $A_\alpha$ similar to \eqref{representation-A-T} for the case when $\alpha$ is in the neighborhood of $0$, we utilize the standard transformation/conjugation operators, see \cite[Chapter II, Section 4.2]{Kato} and  \cite[Chapter 4, Section 1]{DK}. First, we need to show that 
the circle $\partial D(0,\frac{\nu}{2})$ separates $\sigma(A_\alpha)$ into two disjoint spectral subsets for any $\alpha$ in a neighborhood of $0$. For the remainder of the section, we need to assume Hypothesis (H4) in order to ensure that all the spectral and semigroup bounds are $\alpha$-independent. 
Moreover, we recal that from \eqref{representationE0} and Hypothesis (H3) we can immediately infer that $E_0$ is continuous on $(0,\infty)$ in $\mathcal{B}(\bX)$.  
\begin{lemma}\label{new-l3.6}
Assume Hypotheses (H1)--(H4). Then, 
\begin{equation}\label{def-calE-nu}
\cE_\nu:=\Big\{\lambda\in\CC:\mathrm{Re}\lambda\geq-\frac{3\nu}{4},\;|\lambda|\geq\frac{\nu}{4}\Big\}\subset\rho(A_\alpha)\;
\mbox{and}\; \|R(\lambda,A_\alpha)\|\leq\frac{16M_2}{\nu}
\end{equation}
for any $\lambda\in\cE_\nu,\,\alpha\in [0,\eps_0]$, where
\begin{equation}\label{def-eps-0}
\eps_0:=\min\Big\{\frac{\nu}{16M_2\sup_{\alpha\in[0,1]}\|E_0(\alpha)\|+1},1\Big\}
\end{equation}
and $M_2$ is defined in \eqref{l3.5.4}.
\end{lemma}	
\begin{proof}
First, we collect some properties of the linear operator $\tA$. From Remark~\ref{r3.4}(iv) and Lemma~\ref{l3.5}(ii), respectively, we have 
\begin{equation}\label{new-l3.5.1}
\sigma(\tA)\subset \{\lambda\in\CC:\mathrm{Re}\lambda\leq-\nu\}\;\mbox{and}\;\|R(\lambda,\tA)\|\leq\frac{M_2}{\mathrm{Re}\lambda+\frac{7\nu}{8}}\;\;\mbox{whenever}\;\mathrm{Re}\lambda>-\frac{7\nu}{8}.
\end{equation}
In addition, from Lemma~\ref{l3.5}(iii) it follows that the resolvent operator of $A$ has the representation
\begin{equation}\label{new-l3.5.2}
R(\lambda,A)=\begin{bmatrix}
\frac{1}{\lambda}I_{\mathrm{Im}P_0}&0\\
0&R(\lambda,\tA)\end{bmatrix}\;\;\mbox{whenever}\;\mathrm{Re}\lambda>-\nu,\;\lambda\ne0.
\end{equation}
with respect to the decomposition $\bX=\mathrm{Im}P_0\oplus\mathrm{Ker}P_0$. 
From \eqref{new-l3.5.1} and \eqref{new-l3.5.2} it follows that 
\begin{equation}\label{new-l3.5.4}
\|R(\lambda,A)\|\leq\max\Big\{\frac{1}{|\lambda|},\frac{M_2}{\mathrm{Re}\lambda+\frac{7\nu}{8}}\Big\}\leq\max\Big\{\frac{4}{\nu},\frac{8M_2}{\nu}\Big\}=\frac{8M_2}{\nu}\;\mbox{for any}\;\lambda\in\cE_\nu.
\end{equation}
Fix $\lambda\in \cE_\nu$ and $\alpha\in[0,\eps_0]$, where $\eps_0$ is defined in \eqref{def-eps-0}. Then,
\begin{equation}\label{new-l3.5.5}
\lambda I_{\bX}-A_\alpha=\lambda I_{\bX}-A-E(\alpha)=\Big(I_{\bX}-E(\alpha)R(\lambda,A)\Big)(\lambda I_{\bX}-A).
\end{equation}
Moreover, from \eqref{representationE0} and \eqref{new-l3.5.4} we obtain that 
\begin{equation}\label{l3.9.6}
\|E(\alpha)R(\lambda,A)\|\leq \frac{8M_2\alpha}{\nu}\|E_0(\alpha)\|\leq\frac{8M_2\eps_0}{\nu}\sup_{\alpha\in[0,\eps_0]}\|E_0(\alpha)\|\leq\frac{1}{2},
\end{equation}
which implies that $I_{\bX}-E(\alpha)R(\lambda,A)$ is invertible and $\Big\|\big(I_{\bX}-E(\alpha)R(\lambda,A)\big)^{-1}\Big\|\leq 2$. From \eqref{new-l3.5.4} and \eqref{new-l3.5.5} we infer that 
\begin{equation*}
\cE_\nu\subset\rho(A_\alpha)\;\mbox{and}\;\|R(\lambda,A_\alpha)\|\leq\frac{16M_2}{\nu}\;\mbox{for any}\;\lambda\in\cE_\nu,\;\alpha\in[0,\eps_0],
\end{equation*}
proving the lemma.
\end{proof}
Next, we note that assertion \eqref{def-calE-nu} allows us to define the spectral projection
\begin{equation}\label{def-P-alpha}
P_\alpha=\frac{1}{2\pi\rmi}\int_{\partial D(0,\frac{\nu}{2})}R(\lambda,A_\alpha)\rmd\lambda, \quad \alpha\in [0,\eps_0].
\end{equation} 
\begin{remark}\label{r3.6}
Clearly, $P_\alpha$ satisfies the following properties:
\begin{enumerate}
\item[(i)] $\mathrm{Im}P_\alpha$ and $\mathrm{Ker}P_\alpha$ are invariant under $A_\alpha$ and $T_\alpha(t)$ for any $t\geq 0$, $\alpha\in [0,\eps_0]$;
\item[(ii)] $\mathrm{Im}P_\alpha\subset\dom(A)$ for any $\alpha\in [0,\eps_0]$;
\item[(iii)] The linear operator $\tA_\alpha:=(A_\alpha)_{|\dom(A)\cap\mathrm{Ker}P_\alpha}$ is the generator of the semigroup $\{\tT_\alpha(t)\}_{t\geq 0}=\{T(t)_{|\mathrm{Ker}P_\alpha}\}_{t\geq 0}$ on $\mathrm{Ker}P_\alpha$ for any $\alpha\in [0,\eps_0]$;
\item[(iv)] $\sigma(\tA_\alpha)=\sigma(A_\alpha)\setminus D(0,\frac{\nu}{2})$ for any $\alpha\in [0,\eps_0]$.
\end{enumerate}
\end{remark}		
In the next lemma we show that the family of spectral projections $\{P_\alpha\}_{\alpha\in[0,\eps_0]}$ is continuous in the operator norm.
\begin{lemma}\label{new-l3.8}
Assume Hypotheses (H1)--(H4). Then, the function
$\alpha\rightarrow P_\alpha:[0,\eps_0]\to\mathcal{B}(\bX)$ is continuous in the operator norm. Moreover, the following estimate holds
\begin{equation}\label{estimate-P-alpha}
\|P_{\alpha_1}-P_{\alpha_2}\|\leq\frac{128M_2^2}{\nu}\|E(\alpha_1)-E(\alpha_2)\|\quad\mbox{for any}\quad\alpha_1,\alpha_2\in[0,\eps_0].
\end{equation}
\end{lemma}			
\begin{proof}
Since $A_\alpha=A+E(\alpha)$ for any $\alpha\geq 0$, from \eqref{def-calE-nu} we immediately infer that 
\begin{equation}\label{new-l3.8.1}
R(\lambda,A_{\alpha_1})-R(\lambda,A_{\alpha_2})=R(\lambda,A_{\alpha_1})\big(E(\alpha_1)-E(\alpha_2)\big)R(\lambda,A_{\alpha_2})
\end{equation}
for any $\lambda\in\cE_\nu$, $\alpha_1,\alpha_2\in[0,\eps_0]$. From \eqref{def-calE-nu}, \eqref{def-P-alpha} and \eqref{new-l3.8.1} it follows that
\begin{align}\label{new-l3.8.2}
\|P_{\alpha_1}-P_{\alpha_2}\|&=\frac{1}{2\pi}\Big\|\int_{\partial D(0,\frac{\nu}{2})}R(\lambda,A_{\alpha_1})\big(E(\alpha_1)-E(\alpha_2)\big)R(\lambda,A_{\alpha_2})\rmd\lambda\Big\|\nonumber\\
&\leq\frac{\mathrm{length}\big(\partial D(0,\frac{\nu}{2})\big)}{2\pi}\sup_{|\lambda|=\frac{\nu}{2}}\|R(\lambda,A_{\alpha_1})\|\,\|E(\alpha_1)-E(\alpha_2)\|\,\sup_{|\lambda|=\frac{\nu}{2}}\|R(\lambda,A_{\alpha_2})\|\nonumber\\
&\leq\frac{128M_2^2}{\nu}\|E(\alpha_1)-E(\alpha_2)\|\quad\mbox{for any}\quad\alpha_1,\alpha_2\in[0,\eps_0].
\end{align}
The lemma follows shortly from \eqref{new-l3.8.2} and since $E$ is continuous on $[0,\infty)$ in $\mathcal{B}(\bX)$.
\end{proof}		
Following \cite[Chapter 4]{DK} we introduce the operator transformation/conjugation function associated to the family of projections $\{P_\alpha\}_{\alpha\in[0,\eps_0]}$ as follows:
\begin{equation}\label{def-U}
U:[0,\eps_0]\to\mathcal{B}(\bX),\quad U(\alpha)=P_\alpha P_0+(I_\bX-P_\alpha)(I_\bX-P_0).
\end{equation}
The properties of the operator transformation/conjugation function $U$ defined above are collected in the following lemma.
\begin{lemma}\label{new-l3.9}
	Assume Hypotheses (H1)--(H4). Then, the following assertions hold true
\begin{enumerate}
\item[(i)] The function $U$ is continuous on $[0,\eps_0]$ in the $\mathcal{B}(\bX)$ operator norm;
\item[(ii)] $U(\alpha)$  is an invertible operator for any $\alpha\in[0,\eps_1]$, where
\begin{equation}\label{def-eps-1}
\eps_1:=\min\Big\{\frac{\nu}{256M_2^2(16M_2+1)\sup_{\alpha\in[0,1]}\|E_0(\alpha)\|+1},1\Big\}
\end{equation}
and $M_2$ is defined in \eqref{l3.5.4};
\item[(iii)] The function $U(\cdot)^{-1}$ is continuous on $[0,\eps_1]$ in the $\mathcal{B}(\bX)$ operator norm;
\item[(iv)] $P_\alpha=U(\alpha)P_0U(\alpha)^{-1}$ for any $\alpha\in[0,\eps_1]$.  	
\end{enumerate}
\end{lemma}	
\begin{proof}
Assertion (i) follows immediately from \eqref{def-U} and Lemma~\ref{new-l3.8}.

\noindent (ii) Since $P_0$ is a projection one can readily check that 
\begin{align}\label{new-l3.9.1}
U(\alpha)&=I_\bX+P_\alpha P_0+(I_\bX-P_\alpha)(I_\bX-P_0)-P_0^2-(I_\bX-P_0)^2\nonumber\\
&=I_\bX+(P_\alpha-P_0)P_0+\Big(  (I_\bX-P_\alpha)-(I_\bX-P_0)\Big)(I_\bX-P_0)\nonumber\\
&=I_\bX+(P_\alpha-P_0)(2P_0-I_\bX)\;\mbox{for any}\; \alpha\in[0,\eps_0].
\end{align}
Moreover, from \eqref{def-P-0} and Lemma~\ref{new-l3.6} we have
\begin{equation}\label{new-l3.9.2}
\|P_{\alpha}\|=\frac{1}{2\pi}\Big\|\int_{\partial D(0,\frac{\nu}{2})}R(\lambda,A_{\alpha})\rmd\lambda\Big\|\leq\frac{\mathrm{length}\big(\partial D(0,\frac{\nu}{2})\big)}{2\pi}\sup_{|\lambda|=\frac{\nu}{2}}\|R(\lambda,A_{\alpha})\|\leq 8M_2
\end{equation}
for any $\alpha\in[0,\eps_0]$. Since $E(0)=0$ and $\eps_1\leq 1$ from \eqref{representationE0}, \eqref{estimate-P-alpha} and \eqref{new-l3.9.2} we obtain that 
\begin{equation}\label{new-l3.9.3}
\|(P_{\alpha}-P_0)(2P_0-I_\bX)\|\leq\frac{128M_2^2}{\nu}(16M_2+1)\|E(\alpha)\|\leq 
\alpha\frac{128M_2^2(16M_2+1)}{\nu}\sup_{\alpha\in[0,1]}\|E_0(\alpha)\|\leq\frac{1}{2}
\end{equation}
for any $\alpha\in [0,\eps_1]$. From \eqref{new-l3.9.1} and \eqref{new-l3.9.3} we obtain that $U(\alpha)$ is invertible for any $\alpha\in[0,\eps_0]$. Moreover,
\begin{equation}\label{new-l3.9.4}
\|U(\alpha)\|\leq \frac{3}{2},\quad\|U(\alpha)^{-1}\|\leq 2\quad\mbox{for any}\quad \alpha\in[0,\eps_1].
\end{equation}
(iii) From \eqref{new-l3.9.4} we immediately infer that 
\begin{equation}\label{new-l3.9.5}
\|U(\alpha_1)^{-1}-U(\alpha_2)^{-1}\|\leq \|U(\alpha_1)^{-1}\|\,\|U(\alpha_1)-U(\alpha_2)\|\,\|U(\alpha_2)^{-1}\|\leq 4\|U(\alpha_1)-U(\alpha_2)\|
\end{equation}
for any $\alpha_1,\alpha_2\in[0,\eps_1]$. Assertion (iii) follows from (i) and \eqref{new-l3.9.5}.

\noindent (iv) Since $P_\alpha$ is a projection for any $\alpha\in [0,\eps_1]$ from \eqref{def-U} we have  
\begin{equation}\label{new-l3.9.6}
U(\alpha)P_0=P_\alpha P_0=P_\alpha U(\alpha)\quad\mbox{for any}\quad \alpha\in[0,\eps_1].
\end{equation}
Assertion (iv) follows from (iii) and \eqref{new-l3.9.6}.
\end{proof}
When $\bX$ is a Hilbert space, the operator transformation/conjugation function $U$ can be chosen such that $U(\alpha)$ is an unitary operator, see \cite[Formula (4.18), page 102]{Kato}. In this case the estimates \eqref{new-l3.9.4} are simpler. However, since our goal is to find uniform bounds, not necessarily optimal bounds, not having $U(\alpha)$ and its inverse of norm $1$ is not an inconvenience.

We introduce the linear operators $B_\alpha:\dom(B_\alpha)\subseteq\bX\to\bX$, $\alpha\in[0,\eps_1]$ by
\begin{equation}\label{def-B-alpha}
\dom(B_\alpha)=U(\alpha)^{-1}\dom(A),\quad B_\alpha=U(\alpha)^{-1}A_\alpha U(\alpha),\;\alpha\in[0,\eps_1].
\end{equation}
Next, we collect some of the properties of the family of linear operators $B_\alpha$, $\alpha\in[0,\eps_1]$.
\begin{remark}\label{r3.7}
Since $U(\alpha)$ is invertible by Lemma~\ref{new-l3.9}(ii) for any $\alpha\in[0,\eps_1]$ from \eqref{def-B-alpha} we infer that 
\begin{equation}\label{r3.7.1}
\sigma(B_\alpha)=\sigma(A_\alpha),\quad\sigma_{\mathrm{disc}}(B_\alpha)=\sigma_{\mathrm{disc}}(A_\alpha)\quad\mbox{for any}\quad\alpha\in[0,\eps_1].
\end{equation}
Therefore, the circle $\partial D(0,\frac{\nu}{2})$ separates $\sigma(B_\alpha)$ into two disjoint spectral subsets for any $\alpha\in[0,\eps_1]$. 
Moreover, from \eqref{def-B-alpha} we have 
\begin{equation}\label{r3.7.3}	R(\lambda,B_\alpha)=U(\alpha)^{-1}R(\lambda,A_\alpha)U(\alpha)\quad\mbox{for any}\quad\lambda\in\rho(A_\alpha)=\rho(B_\alpha),\;\alpha\in[0,\eps_1].	\end{equation}
From Lemma~\ref{l3.1}, \eqref{r3.7.1} and \eqref{r3.7.3} we obtain 
\begin{equation}\label{r3.7.4}
\Omega_{\ta,\tte}\subseteq\rho(B_\alpha),\quad \|R(\lambda,B_\alpha)\|\leq\|U(\alpha)^{-1}\|\,\|R(\lambda,A_\alpha)\|\,\|U(\alpha)\|\leq \frac{3\tM}{|\lambda-\ta|}
\end{equation}
for any $\lambda\in\Omega_{\ta,\tte}$, $\alpha\in[0,\eps_1]$.
\end{remark}
\begin{lemma}\label{l3.8}
Assume Hypotheses (H1)--(H4). Then, the following assertions hold true:
\begin{enumerate}
\item[(i)] $P_0$ is the spectral projection of $B_\alpha$ associated to the spectral subset of $\sigma(B_\alpha)$ contained in $D(0,\frac{\nu}{2})$ for any $\alpha\in[0,\eps_1]$.
\item[(ii)] $\mathrm{Im}\,P_0\subseteq\dom(B_\alpha)$ for any $\alpha\in[0,\eps_1]$.
\item[(iii)] $\mathrm{Im}\,P_0$ and $\mathrm{Ker}\,P_0$ are invariant under $B_\alpha$ for any $\alpha\in[0,\eps_1]$.
\end{enumerate}	
\end{lemma}	
\begin{proof}
(i) Since $\partial D(0,\frac{\nu}{2})\subset\cE_\nu\subset\rho(A_\alpha)=\rho(B_\alpha)$ for any $\alpha\in[0,\eps_1]\subset [0,\eps_0]$ by Lemma~\ref{new-l3.6}, from \eqref{r3.7.3} and Lemma~\ref{new-l3.9}(iv) it follows that the spectral projection of $B_\alpha$ associated to the spectral subset of $\sigma(B_\alpha)$ contained in $D(0,\frac{\nu}{2})$ is given by
\begin{equation}\label{l3.8.1}
\frac{1}{2\pi\rmi}\int_{\partial D(0,\frac{\nu}{2})}R(\lambda,B_\alpha)\rmd\lambda=\frac{1}{2\pi\rmi}\int_{\partial D(0,\frac{\nu}{2})}U(\alpha)^{-1}R(\lambda,A_\alpha)U(\alpha)\rmd\lambda=U(\alpha)^{-1}P_\alpha U(\alpha)=P_0.
\end{equation}
for any $\alpha\in[0,\eps_1]$. Assertions (ii) and (iii) follow shortly, cf.
\cite{EN,lunardi,P}.
\end{proof}	
We define the linear operators $\tB_\alpha:\dom(\tB_\alpha)\subseteq\mathrm{Ker}\,P_0\to\mathrm{Ker}\,P_0$ and $K_\alpha:\mathrm{Im}\,P_0\to\mathrm{Im}\,P_0$ by
\begin{equation}\label{def-tilde-B-K}
\dom(\tB_\alpha)=\dom(B_\alpha)\cap\mathrm{Ker}P_0,\; \tB_\alpha=(B_\alpha)_{|\dom(B_\alpha)\cap\mathrm{Ker}P_0},\; K_\alpha=(B_\alpha)_{|\mathrm{Im}P_0},\;\alpha\in[0,\eps_1].
\end{equation}
In addition, we denote by $\{S_\alpha(t)\}_{t\geq 0}$ and $\{\tS_\alpha(t)\}_{t\geq 0}$, the semigroups generated by $B_\alpha$ and $\tB_\alpha$, $\alpha\in[0,\eps_1]$, respectively. The advantage of working with the family of operators $B_\alpha$, $\alpha\in[0,\eps_1]$, is that its spectral projection $P_0$
associated to the spectral subset of $\sigma(B_\alpha)$ contained in $D(0,\frac{\nu}{2})$ is \textit{independent} on $\alpha\in[0,\eps_1]$. 
From \eqref{r3.7.4} we see that the semigroups $\{S_\alpha(t)\}_{t\geq 0}$ and $\{\tS_\alpha(t)\}_{t\geq 0}$ are analytic, moreover, the following representation holds:
\begin{equation}\label{representation-S}
T_\alpha(t)=U(\alpha)S_\alpha(t)U(\alpha)^{-1},\; S_\alpha(t)=\begin{bmatrix}
e^{tK_\alpha}&0\\
0&\tS_\alpha(t)\end{bmatrix},\;\mbox{for any}\; t\geq 0,\,\alpha\in[0,\eps_1],
\end{equation}
with respect to the decomposition $\bX=\mathrm{Im}P_0\oplus\mathrm{Ker}P_0$.

In the next lemma we study the spectrum of the linear operator $\tB_\alpha$, in particular we estimate  $\sup\mathrm{Re}\sigma(\tB_\alpha)$ for $\alpha\in [0,\eps_1]$. 
\begin{lemma}\label{l3.9}
Assume Hypotheses (H1)--(H4). Then, the following assertions hold true:
\begin{enumerate}
\item[(i)] $\sigma(\tB_\alpha)\subseteq\{\lambda\in\CC:\mathrm{Re}\lambda\leq-\frac{\nu}{2}\}$ for any $\alpha\in[0,\eps_1]$;
\item[(ii)] $\|R(\lambda,\tB_\alpha)\|\leq \frac{96M_2}{\nu}$ whenever $\mathrm{Re}\lambda\geq -\frac{\nu}{2}$, $\alpha\in [0,\eps_1]$.
\end{enumerate}	
\end{lemma}	
\begin{proof}
(i) From Lemma~\ref{new-l3.6}, Lemma~\ref{l3.8}, \eqref{def-tilde-B-K} and \eqref{r3.7.1}  we have  
\begin{equation}\label{l3.9.8}
\sigma(\tB_\alpha)=\sigma(B_\alpha)\setminus D\big(0,\frac{\nu}{2}\big)=\sigma(A_\alpha)\setminus D\big(0,\frac{\nu}{2}\big)\subseteq(\CC\setminus\cE_\nu)\setminus D\big(0,\frac{\nu}{2}\big)\subseteq\big\{\lambda\in\CC:\mathrm{Re}\lambda\leq-\frac{\nu}{2}\big\}.
\end{equation}
for any $\alpha\in[0,\eps_1]$, proving (i).

\noindent (ii) First, we note that from Lemma~\ref{new-l3.6}, \eqref{new-l3.9.4} and \eqref{r3.7.3} it follows that 
\begin{equation}\label{l3.9.9}
\|R(\lambda,\tB_\alpha)\|\leq\|R(\lambda,B_\alpha)\|=\|U(\alpha)^{-1}R(\lambda,A_\alpha)U(\alpha)\|\leq 3\|R(\lambda,A_\alpha)\|\leq\frac{48M_2}{\nu}
\end{equation}  
for any $\lambda\in\cE_\nu$, $\alpha\in[0,\eps_1]$. To prove the lemma we need to prove the estimate from (ii) for the case when $\lambda\in D(0,\frac{\nu}{4})$. 
From Lemma~\ref{new-l3.9}(iv) one can readily check that
\begin{equation}\label{l3.9.10}
\mathrm{Ker}P_\alpha=U(\alpha)\mathrm{Ker}P_0\;\mbox{for any}\;\alpha\in[0,\eps_1].
\end{equation}
We infer that the linear operator $\tU(\alpha):=U(\alpha)_{|\mathrm{Ker}P_0}$is bounded, invertible from $\mathrm{Ker}P_0$ to $\mathrm{Ker}P_\alpha$, with bounded inverse $\tU(\alpha)^{-1}=U(\alpha)^{-1}_{|\mathrm{Ker}P_\alpha}$, for any $\alpha\in[0,\eps_1]$. From Lemma~\ref{new-l3.9}(i), \eqref{def-B-alpha} and \eqref{l3.9.10} we obtain that
\begin{equation}\label{l3.9.11}
U(\alpha)\big(\dom(B_\alpha)\cap\mathrm{Ker}P_0\big)=\dom(A_\alpha)\cap\mathrm{Ker}P_\alpha=\dom(A)\cap\mathrm{Ker}P_\alpha\;\mbox{for any}\;\alpha\in[0,\eps_1],
\end{equation}
which implies that 
\begin{equation}\label{l3.9.12}
\tB_\alpha=\tU(\alpha)^{-1}\tA_\alpha\tU(\alpha)\;\mbox{for any}\;\alpha\in[0,\eps_1].
\end{equation}
Moreover, from \eqref{new-l3.9.4} and \eqref{l3.9.12} we have
\begin{align}\label{l3.9.13}
\|R(\lambda,\tB_\alpha)&=\|\tU(\alpha)^{-1}R(\lambda,\tA_\alpha)\tU(\alpha)\|\leq\|U(\alpha)^{-1}_{|\mathrm{Ker}P_\alpha}\|\,\|R(\lambda,\tA_\alpha)\|\,\|U(\alpha)_{|\mathrm{Ker}P_0}\|\nonumber\\
&\leq \|U(\alpha)^{-1}\|\,\|R(\lambda,\tA_\alpha)\|\,\|U(\alpha)\|\leq 3\|R(\lambda,\tA_\alpha)\|
\end{align} 
for any $\lambda\in\rho(\tB_\alpha)=\rho(\tA_\alpha)$ and $\alpha\in[0,\eps_1]$. 
To estimate $\|R(\lambda,\tA_\alpha)\|$ for $\lambda\in D(0,\frac{\nu}{4})$ and $\alpha\in[0,\eps_1]$ we need to find a contour integral representation of $R(\lambda,\tA_\alpha)$. We consider the function $\tF:\CC\setminus\partial D(0,\frac{\nu}{2})\times[0,\eps_1]\to\mathcal{B}(\bX)$ defined by 
\begin{equation}\label{l3.9.14}
\tF(\lambda,\alpha):=\frac{1}{2\pi\rmi}\int_{\partial D(0,\frac{\nu}{2})}\frac{1}{\lambda-\zeta}R(\zeta,A_\alpha)\rmd\zeta,
\end{equation}
c.f. \cite[Section III.6.5]{Kato}. From Remark~\ref{r3.6}(i) we have $P_\alpha R(\zeta,A_\alpha)=R(\zeta,A_\alpha)P_\alpha$ for any $\zeta\in\partial D(0,\frac{\nu}{2})$ and $\alpha\in[0,\eps_1]$, which implies that 
\begin{equation}\label{l3.9.15}
P_\alpha \tF(\lambda,\alpha)=\tF(\lambda,\alpha)P_\alpha\;\mbox{for any}\;\lambda\in\CC\setminus\partial D(0,\frac{\nu}{2}),\;\alpha\in[0,\eps_1].
\end{equation}
Since $A_\alpha R(\zeta,A_\alpha)=\zeta R(\zeta,A_\alpha)-I_{\bX}$ for any $\zeta\in\partial D(0,\frac{\nu}{2})$ and $A_\alpha$ is a closed linear operator we obtain that $\mathrm{Im}\,\tF(\lambda,\alpha)\subseteq\dom(A_\alpha)$ and 
\begin{align}\label{l3.9.16}
A_\alpha \tF(\lambda,\alpha)&=\frac{1}{2\pi\rmi}\int_{\partial D(0,\frac{\nu}{2})}\frac{1}{\lambda-\zeta}\big(\zeta R(\zeta,A_\alpha)-I_{\bX}\big)\rmd\zeta\nonumber\\
&=\frac{1}{2\pi\rmi}\int_{\partial D(0,\frac{\nu}{2})}\Big(\frac{\lambda}{\lambda-\zeta}-1\Big)R(\zeta,A_\alpha)\rmd\zeta+\frac{1}{2\pi\rmi}\Bigg(\int_{\partial D(0,\frac{\nu}{2})}\frac{\rmd\zeta}{\zeta-\lambda}\Bigg)I_{\bX}\nonumber\\
&=\lambda\tF(\lambda,\alpha)-P_\alpha+\frac{1}{2\pi\rmi}\Bigg(\int_{\partial D(0,\frac{\nu}{2})}\frac{\rmd\zeta}{\zeta-\lambda}\Bigg)I_{\bX}
\end{align}
for any $\lambda\in\CC\setminus\partial D(0,\frac{\nu}{2})$ and $\alpha\in[0,\eps_1]$. From \eqref{l3.9.15} and \eqref{l3.9.16} it follows that 
$\tF(\lambda,\alpha)\mathrm{Ker}P_\alpha\subseteq(\dom(A_\alpha)\cap\mathrm{Ker}P_\alpha)=\dom(\tA_\alpha)$ and 
\begin{equation}\label{l3.9.17}
\tA_\alpha\tF(\lambda,\alpha)x=A_\alpha\tF(\lambda,\alpha)x=\lambda\tF(\lambda,\alpha)x-P_\alpha x+\frac{1}{2\pi\rmi}\Bigg(\int_{\partial D(0,\frac{\nu}{2})}\frac{\rmd\zeta}{\zeta-\lambda}\Bigg)x=\lambda\tF(\lambda,\alpha)x+x
\end{equation}
for any $x\in\mathrm{Ker}P_\alpha$, $\lambda\in D(0,\frac{\nu}{4})$ and $\alpha\in[0,\eps_1]$. Since $D(0,\frac{\nu}{4})\subset D(0,\frac{\nu}{2})\subseteq\rho(\tA_\alpha)$ for any $\alpha\in[0,\eps_1]$, by Remark~\ref{r3.6}(iv), we conclude that $R(\lambda,\tA_\alpha)=-\tF(\lambda,A_\alpha)_{|\mathrm{Ker}P_\alpha}$ for any  $\lambda\in D(0,\frac{\nu}{4})$ and $\alpha\in[0,\eps_1]$. Since $D(0,\frac{\nu}{2})\subset\cE_\nu$ and $|\lambda-\zeta|\geq|\zeta|-|\lambda|\geq\frac{\nu}{4}$ for any $\lambda\in D(0,\frac{\nu}{4})$ and $\zeta\in\partial D(0,\frac{\nu}{2})$, from Lemma~\ref{new-l3.6} and \eqref{l3.9.14} we infer that 
\begin{align}\label{l3.9.18}
\|R(\lambda,\tA_\alpha)\|&=\|\tF(\lambda,A_\alpha)_{|\mathrm{Ker}P_\alpha}\|\leq\|\tF(\lambda,\alpha)\|=\Big\|\frac{1}{2\pi\rmi}\int_{\partial D(0,\frac{\nu}{2})}\frac{1}{\lambda-\zeta}R(\zeta,A_\alpha)\rmd\zeta\Big\|\nonumber\\
&\leq\frac{\mathrm{length}\big(\partial D(0,\frac{\nu}{2})\big)}{2\pi}\sup_{|\zeta|=\frac{\nu}{2}}\|R(\zeta,A_\alpha)\|\sup_{|\zeta|=\frac{\nu}{2}}\frac{1}{|\lambda-\zeta|}\leq
\frac{32M_2}{\nu}
\end{align}
for any  $\lambda\in D(0,\frac{\nu}{4})$ and $\alpha\in[0,\eps_1]$. Finally, from 
\eqref{l3.9.13} and \eqref{l3.9.18} we conclude that 
\begin{equation}\label{l3.9.19}
\|R(\lambda,\tB_\alpha)\|\leq\frac{96M_2}{\nu}\;\mbox{for any}\;\lambda\in D(0,\frac{\nu}{4}),\;\alpha\in[0,\eps_1],
\end{equation}
proving the lemma.
\end{proof}
We are now ready to prove that the semigroup $\{\tS_\alpha(t)\}_{t\geq 0}$, generated by $\tB_\alpha$, is uniformly exponentially stable for $\alpha\in[0,\eps_1]$. 
\begin{lemma}\label{l3.10}
Assume Hypotheses (H1)--(H4). Then, the following estimate holds
\begin{equation}\label{est-tilde-S-alpha}
\|\tS_\alpha(t)\|\leq M_3e^{-\frac{\nu}{2}t}\quad\mbox{for any}\quad t\geq 0,\;\alpha\in[0,\eps_1]\;\mbox{where},
\end{equation}
\begin{equation}\label{M-3}
M_3=\max\Bigg\{\Big(\frac{3\tM}{\pi\ta}+\frac{96M_2}{\pi\nu}(\ta+\frac{\nu}{2})|\tan(\frac{\tte}{2}+\frac{\pi}{4})|\Big),\frac{\tM\Big(e(6\tte+3\pi)-12\sec(\frac{\tte}{2}+\frac{\pi}{4})\Big)}{4\pi}e^{(\ta+\frac{\nu}{2})}\Bigg\},
\end{equation}
and $M_2$ is defined in \eqref{l3.5.4}.
\end{lemma}	
\begin{proof} The proof of the lemma is similar to the proof of Lemma~\ref{l3.2}, the main tool is the estimate \eqref{V-estimate} of Lemma~\ref{l2.4}. We fix $\alpha\in[0,\eps_1]$ and $\mu\in(-\frac{\nu}{2},0)$. Hence, $\mu>\mathrm{s}(\tB_\alpha)$ by Lemma~\ref{l3.9}. Also, we set $\varphi=\frac{\tte}{2}+\frac{\pi}{4}\in(\frac{\pi}{2},\tte)$ and $\widetilde{b}=(\ta-\mu)|\tan\varphi|$. 
	
First, from \eqref{r3.7.4} and \eqref{def-tilde-B-K} we immediately conclude that $\tB_\alpha$ is sectorial, moreover 	
\begin{equation}\label{l3.10.1}
\Omega_{\ta,\tte}\subseteq\rho(B_\alpha)\subseteq\rho(\tB_\alpha),\quad \|R(\lambda,\tB_\alpha)\|\leq\|R(\lambda,B_\alpha)\|\leq \frac{3\tM}{|\lambda-\ta|}\;\mbox{for any}\;\lambda\in\Omega_{\ta,\tte}.
\end{equation}  
In addition, from \eqref{def-V} and Lemma~\ref{l3.9}(ii) we obtain 
\begin{equation}\label{l3.10.2}
\|\cV_{\tB_\alpha}(t,\mu,\varphi)\|=\Big\|\int_{-\widetilde{b}}^{\widetilde{b}}e^{\rmi st}R(\mu+\rmi s,\tB_\alpha)\rmd s\Big\|\leq\frac{192M_2\widetilde{b}}{\nu}=\frac{192M_2(\ta-\mu)|\tan(\frac{\tte}{2}+\frac{\pi}{4})|}{\nu}
\end{equation}	
for any $t>0$. Applying Lemma~\ref{l2.4}, from \eqref{l3.10.1} and \eqref{l3.10.2} we infer that 
\begin{equation}\label{l3.10.3}
\|\tS_\alpha(t)\|\leq\frac{3\tM e^{\mu t}}{\pi(\ta-\mu)t}+\frac{96M_2(\ta-\mu)|\tan(\frac{\tte}{2}+\frac{\pi}{4})|e^{\mu t}}{\pi\nu}\leq\Big(\frac{3\tM}{\pi\ta}+\frac{96M_2}{\pi\nu}(\ta+\frac{\nu}{2})|\tan(\frac{\tte}{2}+\frac{\pi}{4})|\Big)e^{\mu t}
\end{equation}
for any $t\geq 1$. From Lemma~\ref{l2.3}, Lemma~\ref{l3.1}, \eqref{new-l3.9.4} and \eqref{representation-S}  it follows that
\begin{align}\label{l3.10.4}
\|\tS_\alpha(t)\|&\leq\|S_\alpha(t)\|=\|U(\alpha)^{-1} T_\alpha(t)U(\alpha)  
\|\leq \|U(\alpha)^{-1}\|\,\|\|T_\alpha(t)\|\,\|U(\alpha)\|\
\leq 3\|T_\alpha(t)\|\nonumber\\
&\leq\frac{3\tM(e\varphi-\sec\varphi)}{\pi}e^{\ta t}=\frac{\tM\Big(e(6\tte+3\pi)-12\sec(\frac{\tte}{2}+\frac{\pi}{4})\Big)}{4\pi}e^{(\ta+\frac{\nu}{2})t} e^{-t\frac{\nu}{2}}\nonumber\\
&\leq \frac{\tM\Big(e(6\tte+3\pi)-12\sec(\frac{\tte}{2}+\frac{\pi}{4})\Big)}{4\pi}e^{(\ta+\frac{\nu}{2})} e^{-t\frac{\nu}{2}}\;\mbox{for any}\;t\in[0,1].
\end{align}
Passing to the limit as $\mu\to-\frac{\nu}{2}$ in \eqref{l3.10.3}, we obtain the estimate \eqref{est-tilde-S-alpha} from \eqref{l3.10.3} and \eqref{l3.10.4}.
\end{proof}	
Next, we focus our attention on the family of operators $K_\alpha=(B_\alpha)_{|\mathrm{Im}P_0}$, $\alpha\in[0,\eps_1]$. First, we show that it depends continuously on $\alpha\in[0,\eps_1]$. Moreover, we look for a representation of $K_\alpha$ in a neighborhood of $0$.
\begin{lemma}\label{l3.11}
Assume Hypotheses (H1)--(H4). Then, the following assertions hold true:
\begin{enumerate}
\item[(i)] The operator-valued function 
$\alpha\to K_\alpha :[0,\eps_1]\to\mathcal{B}(\mathrm{Im}P_0)$is continuous;
\item[(ii)]	$K_\alpha=\alpha \Big(P_0U(\alpha)^{-1}G(\alpha)U(\alpha)\Big)_{|\mathrm{Im}P_0}$ for any $\alpha\in[0,\eps_1]$, where $G:[0,\eps_1]\to\mathcal{B}(\bX)$ defined by
\begin{equation}\label{def-G}
G(\alpha)=\frac{1}{2\pi\rmi}\int_{\partial D(0,\frac{\nu}{2})}\lambda R(\lambda,A_\alpha)E_0(\alpha)R(\lambda,A)\rmd\lambda;
\end{equation}
\item[(iii)] The function $G$ is bounded. Moreover,
\begin{equation}\label{bound-G}
\|G(\alpha)\|\leq 64M_2^2\sup_{\alpha\in[0,1]}\|E_0(\alpha)\|\quad\mbox{for any}\quad\alpha\in [0,\eps_1].
\end{equation}
\end{enumerate}
\end{lemma}	
\begin{proof}
(i) From Lemma~\ref{new-l3.9}(iv) and \eqref{def-tilde-B-K} we have 
\begin{equation}\label{l3.11.1}
K_\alpha x=P_0K_\alpha x=P_0U(\alpha)^{-1}A_\alpha U(\alpha)x=P_0U(\alpha)^{-1}A_\alpha U(\alpha) P_0x=P_0U(\alpha)^{-1}A_\alpha P_\alpha U(\alpha)x
\end{equation}
for any $x\in\mathrm{Im}P_0$ and $\alpha\in[0,\eps_1]$. Since $A_\alpha R(\lambda,A_\alpha)=\lambda R(\lambda,A_\alpha)-I_{\bX}$ for any $\lambda\in\rho(A_\alpha)$, $\mathrm{Im}P_\alpha\subset\dom(A_\alpha)$, by Remark~\ref{r3.6}(ii), from \eqref{def-P-alpha} we obtain that 
\begin{align}\label{l3.11.2}
A_\alpha P_\alpha&=\frac{1}{2\pi\rmi}\int_{\partial D(0,\frac{\nu}{2})}A_\alpha R(\lambda,A_\alpha)\rmd\lambda=\frac{1}{2\pi\rmi}\int_{\partial D(0,\frac{\nu}{2})}\big(\lambda R(\lambda,A_\alpha)-I_{\bX}\big)\rmd\lambda\nonumber\\
&=\frac{1}{2\pi\rmi}\int_{\partial D(0,\frac{\nu}{2})}\lambda R(\lambda,A_\alpha)\rmd\lambda
\end{align}
for any $\alpha\in[0,\eps_1]$. Since the operator-valued function $E$ is continuous on $[0,\eps_1]$, by Hypothesis (H3),  from Lemma~\ref{new-l3.6} we infer that the function
\begin{equation}\label{l3.11.3}
(\lambda,\alpha)\to R(\lambda,A_\alpha):\cE_\nu\times[0,\eps_1] \to\mathcal{B}(\bX)\;\mbox{is continuous}.
\end{equation}
Since $\partial D(0,\frac{\nu}{2})\subset\cE_\nu$, by \eqref{def-calE-nu}, from \eqref{l3.11.2} and \eqref{l3.11.3} we conclude that 
\begin{equation}\label{l3.11.4}
\alpha\to A_\alpha P_\alpha :[0,\eps_1]\to\mathcal{B}(\bX)\;\mbox{is continuous}.
\end{equation}
Assertion (i) follows shortly from Lemma~\ref{new-l3.9}(i) and (iii), \eqref{l3.11.1} and \eqref{l3.11.4}. 

\noindent (ii) We recall that from Lemma~\ref{l3.5}(i) and \eqref{l3.5.1} we have
\begin{equation}\label{l3.12.1}
0=AP_0=\frac{1}{2\pi\rmi}\int_{\partial D(0,\frac{\nu}{2})}\big(\lambda R(\lambda,A)-I_{\bX}\big)\rmd\lambda=\frac{1}{2\pi\rmi}\int_{\partial D(0,\frac{\nu}{2})}\lambda R(\lambda,A)\rmd\lambda.
\end{equation}
From \eqref{representationE0}, \eqref{l3.11.2} and \eqref{l3.12.1} it follows that 
\begin{align}\label{l3.12.3}
A_\alpha P_\alpha&=\frac{1}{2\pi\rmi}\int_{\partial D(0,\frac{\nu}{2})}\lambda R(\lambda,A_\alpha)\rmd\lambda=\frac{1}{2\pi\rmi}\int_{\partial D(0,\frac{\nu}{2})}\lambda \Big(R(\lambda,A_\alpha)-R(\lambda,A)\Big)\rmd\lambda\nonumber\\	&=\frac{1}{2\pi\rmi}\int_{\partial D(0,\frac{\nu}{2})}\lambda R(\lambda,A_\alpha)E(\alpha)R(\lambda,A)\rmd\lambda=\frac{\alpha}{2\pi\rmi}\int_{\partial D(0,\frac{\nu}{2})}\lambda R(\lambda,A_\alpha)E_0(\alpha)R(\lambda,A)\rmd\lambda.
\end{align}
for any $\alpha\in[0,\eps_1]$, proving (ii).

\noindent (iii) From \eqref{def-calE-nu} one can readily check that
\begin{align}\label{l3.15.4}
\|G(\alpha)\|&=\Big\| \frac{1}{2\pi\rmi}\int_{\partial D(0,\frac{\nu}{2})}\lambda R(\lambda,A_\alpha)E_0(\alpha)R(\lambda,A)\rmd\lambda \Big\|\nonumber\\
&\leq\frac{\mathrm{length}\big(\partial D(0,\frac{\nu}{2})\big)}{2\pi}     
\sup_{|\lambda|=\frac{\nu}{2}}\|\lambda R(\lambda,A_\alpha)\|\sup_{\alpha\in[0,1]}\|E_0(\alpha)\|
\sup_{|\lambda|=\frac{\nu}{2}}\| R(\lambda,A)\|;\nonumber\\
&\leq 64M_2^2\sup_{\alpha\in[0,1]}\|E_0(\alpha)\|\quad\mbox{for any}\quad\alpha\in [0,\eps_1],
\end{align}
proving the lemma.\end{proof}
The previous lemma shows that $K_\alpha$ is of order $\mathcal{O}(\alpha)$ in a neighborhood of $0$. In addition, we note that the operator valued function $G$ is not necessarily continuous at $0$. Therefore, to prove our main result we assume Hypothesis (H5).
\begin{lemma}\label{l3.12}
Assume Hypotheses (H1)--(H5). Then, the following assertion hold true:
\begin{enumerate}
\item[(i)] The operator valued function $G$ defined in \eqref{def-G} is continuous in the $\mathcal{B}(\bX)$ operator norm;
\item[(ii)] 
$P_0G(0)_{|\mathrm{Im}P_0}=P_0E_0(0)_{|\mathrm{Im}P_0}$.
 \item[(iii)] $\|K_\alpha-\alpha P_0E_0(0)_{|\mathrm{Im}P_0}\|\leq \frac{8M_2\nu}{\eps_1^2}\alpha^2+768M_2^2\alpha r(\alpha)$ for any $\alpha\in [0,\eps_1]$. 
\end{enumerate}
\end{lemma}
\begin{proof}
From \eqref{representationE0} and \eqref{E0-bound} and since $E$ is continuous in the $\mathcal{B}(\bX)$ operator norm by Hypothesis (H3), we infer that $E_0$ is continuous on $[0,\infty)$ in the $\mathcal{B}(\bX)$ operator norm. Moreover, 
$U$ and $U^{-1}$  are continuous on $[0,\eps_1]$, by Lemma~\ref{new-l3.9}.
Assertion (i) follows shortly from \eqref{l3.11.1} and \eqref{l3.11.4}.
	
\noindent (ii)	First, we note that from Lemma~\ref{l3.5}(iii) we have $\big(R(\lambda,A)\big)_{|\mathrm{Im}P_0}=\frac{1}{\lambda}I_{\mathrm{Im}P_0}$ for any $\lambda\in\partial D(0,\frac{\nu}{2})$. Hence, from \eqref{def-G} it follows that
\begin{align}\label{l3.12.4}
P_0G(0)_{|\mathrm{Im}P_0}&=\frac{1}{2\pi\rmi}P_0\int_{\partial D(0,\frac{\nu}{2})} R(\lambda,A)E_0(0)\big(\lambda R(\lambda,A)_{|\mathrm{Im}P_0}\big)\rmd\lambda\nonumber\\
&=\frac{1}{2\pi\rmi}P_0\int_{\partial D(0,\frac{\nu}{2})}R(\lambda,A)E_0(0)_{|\mathrm{Im}P_0}\rmd\lambda=P_0E_0(0)_{|\mathrm{Im}P_0}.
\end{align}
To prove (iii) we need to a long but standard series of estimates of all the functions involved in formula \eqref{def-G}. For completeness we give the details below. Since $E(0)=0$ from \eqref{representationE0}, \eqref{def-calE-nu}, \eqref{new-l3.8.1} and \eqref{def-eps-1} we have 
\begin{equation}\label{l3.15.1}
\|R(\lambda,A_\alpha)-R(\lambda,A)\|\leq \|R(\lambda,A_\alpha)\|\,\|E(\alpha)\|\,\|R(\lambda,A)\|\leq\alpha\frac{128M_2^2}{\nu^2}\sup_{\alpha\in[0,1]}\|E_0(\alpha)\|\leq \frac{\alpha}{2\eps_1\nu}
\end{equation}
for any $\lambda\in\cE_\nu$, $\alpha\in[0,\eps_1]$. From \eqref{E0-bound}, \eqref{def-calE-nu} and \eqref{l3.15.1} we obtain  
\begin{align}\label{l3.15.2}
\|R(\lambda,A_\alpha)E_0(\alpha)-R(\lambda,A)E_0(0)\|&\leq
\Big\|\Big(R(\lambda,A_\alpha)-R(\lambda,A)\Big)E_0(\alpha)\Big\|+\Big\|R(\lambda,A)\Big(E_0(\alpha)-E_0(0)\Big)\Big\|\nonumber\\
&\leq \frac{\alpha}{2\eps_1\nu}\sup_{\alpha\in[0,1]}\|E_0(\alpha)\|+\frac{16M_2}{\nu}r(\alpha)
\end{align}
for any $\lambda\in\cE_\nu$, $\alpha\in[0,\eps_1]$. 
Using again \eqref{def-calE-nu}, from \eqref{def-G} and \eqref{l3.15.2} it follows that 
\begin{align}\label{l3.15.3}
\|G(\alpha)-G(0)\|&=\Big\| \frac{1}{2\pi\rmi}\int_{\partial D(0,\frac{\nu}{2})}\lambda\Big(R(\lambda,A_\alpha)E_0(\alpha)-R(\lambda,A)E_0(0)\Big)R(\lambda,A)\rmd\lambda \Big\|\nonumber\\
&\leq\frac{\mathrm{length}\big(\partial D(0,\frac{\nu}{2})\big)}{2\pi}     
\sup_{|\lambda|=\frac{\nu}{2}}\|\lambda R(\lambda,A)\|\sup_{|\lambda|=\frac{\nu}{2}}\|R(\lambda,A_\alpha)E_0(\alpha)-R(\lambda,A)E_0(0)\|\nonumber\\
&\leq\frac{2M_2\alpha}{\eps_1}\sup_{\alpha\in[0,1]}\|E_0(\alpha)\|
+64M_2^2r(\alpha)\quad\mbox{for any}\quad\alpha\in [0,\eps_1].
\end{align}
Moreover, from \eqref{new-l3.9.1}, \eqref{new-l3.9.3} and \eqref{new-l3.9.5} we have 
\begin{align}\label{l3.15.5}
&\|U(\alpha)-I_\bX\|=\|(P_{\alpha}-P_0)(2P_0-I_\bX)\|\leq
\alpha\frac{128M_2^2(16M_2+1)}{\nu}\sup_{\alpha\in[0,1]}\|E_0(\alpha)\|\leq\frac{\alpha}{2\eps_1}\nonumber\\
&\|U(\alpha)^{-1}-I_\bX\|\leq 2\|U(\alpha)-I_\bX\|\leq\frac{\alpha}{\eps_1}\quad\mbox{for any}\quad\alpha\in [0,\eps_1].
\end{align}
From \eqref{def-eps-1}, \eqref{new-l3.9.4}, \eqref{l3.15.4} and \eqref{l3.15.1}--\eqref{l3.15.5} we infer that 
\begin{align}\label{l3.15.6}
&\|U(\alpha)^{-1}G(\alpha)U(\alpha)-G(0)\|\leq\Big\|\Big(U(\alpha)^{-1}-I_\bX\Big)G(\alpha)U(\alpha)\Big\|+\|G(\alpha)U(\alpha)-G(0)\|\nonumber\\
&\qquad\leq \frac{96M_2^2\alpha}{\eps_1}\sup_{\alpha\in[0,1]}\|E_0(\alpha)\|+\Big\|\Big(G(\alpha)-G(0)\Big)U(\alpha)\Big\|
+\Big\|G(0)\Big(U(\alpha)-I_\bX\Big)\Big\|\nonumber\\
&\qquad\leq\frac{96M_2^2\alpha}{\eps_1}\sup_{\alpha\in[0,1]}\|E_0(\alpha)\|+\frac{3M_2\alpha}{\eps_1}\sup_{\alpha\in[0,1]}\|E_0(\alpha)\|
+96M_2^2r(\alpha)+\frac{32M_2^2\alpha}{\eps_1}\sup_{\alpha\in[0,1]}\|E_0(\alpha)\|\nonumber\\
&\qquad=\alpha\frac{128M_2^2+3M_2}{\eps_1}\sup_{\alpha\in[0,1]}\|E_0(\alpha)\|+96M_2^2r(\alpha)\leq\frac{\alpha\nu}{\eps_1^2}+96M_2^2r(\alpha)\quad\mbox{for any}\quad\alpha\in [0,\eps_1].
\end{align}
Assertion (iii) follows from (i), \eqref{new-l3.9.2}, \eqref{l3.12.4} and \eqref{l3.15.6}.
\end{proof}
\begin{lemma}\label{new-l3.16}
Assume Hypotheses (H1)--(H5). Then, 
\begin{equation}\label{stable-trace-K}
\sup\mathrm{Re}\sigma(P_0E_0(0)_{|\mathrm{Im}P_0})\leq -q_1,
\end{equation}
where the constant $q_1>0$ was introduced in Hypothesis (H3)(ii).
\end{lemma}
\begin{proof}
We denote by $G_0:[0,\eps_1]\to\mathcal{B}(\mathrm{Im}P_0)$ the function defined by \begin{equation}\label{def-G-0}
G_0(\alpha)=\Big(P_0U(\alpha)^{-1}G(\alpha)U(\alpha)\Big)_{|\mathrm{Im}P_0},
\end{equation}
and introduced in Lemma~\ref{l3.12}(i). Fix $\lambda_0\in\sigma(P_0E_0(0)_{|\mathrm{Im}P_0})$. From Lemma~\ref{l3.12}(ii) we have  $\lambda_0\in\sigma(P_0G(0)_{|\mathrm{Im}P_0})=\sigma(G_0(0))$. Since $G$ is continuous on $[0,\eps_1]$ we conclude that $G_0$ is continuous on $[0,\eps_1]$. 
Using the semi-continuity property of the spectrum of bounded linear operators in finite dimensional spaces, there exist two sequences $\{\alpha_n\}_{n\geq 1}$ and $\{\widetilde{\lambda}_n\}_{n\geq 1}$ such that 
\begin{equation}\label{new-l3.16.1}
\alpha_n\to0,\;\widetilde{\lambda}_n\to\lambda_0\;\mbox{as}\;n\to\infty,\quad\alpha_n\in(0,\eps_1),\;\widetilde{\lambda}_n\in\sigma(G_0(\alpha_n))\;\mbox{for any}\;n\geq 1.
\end{equation}
Since $K_\alpha=\alpha G_0(\alpha)$ for any $\alpha\in[0,\eps_1]$, by Lemma~\ref{l3.12}(i), from Hypothesis (H3)(i),  \eqref{r3.7.1}, \eqref{def-tilde-B-K} and \eqref{new-l3.16.1} it follows that 
\begin{equation}\label{l3.12.7}
\alpha_n\widetilde{\lambda}_n\in\sigma(K_{\alpha_n})\subseteq\sigma(B_{\alpha_n})=\sigma(A_{\alpha_n})\subset\{\lambda\in\CC:\mathrm{Re}\lambda\leq -q(\alpha_n)\}\;\;\mbox{for any}\;\; n\geq 1.
\end{equation}
From \eqref{new-l3.16.1} we have there exists $n_0\geq 1$ such that $\alpha_n\in(0,q_2)$ for any $n\geq n_0$.
Since $q(\alpha)=q_1\alpha$ for any $\alpha\in[0,q_2]$, from \eqref{l3.12.7} we obtain  
\begin{equation}\label{l3.12.8}
\mathrm{Re}\widetilde{\lambda}_n\leq -q_1\;\;\mbox{for any}\;\;n\geq n_0.
\end{equation}
Passing to the limit as $n\to\infty$ yields $\mathrm{Re}\lambda_0\leq -q_1$, proving the lemma.	
\end{proof}
We are now ready to estimate the norm of the semigroup generated by $K_\alpha$ for $\alpha$ in a neighborhood of $0$. To formulate the result, we need to point out a couple of immediate consequences of Hypothesis (H5) and Lemma~\ref{new-l3.16}.
\begin{remark}\label{r3.13}
We assumed in Hypothesis (H5) that the function $r:[0,\infty)\to[0,\infty)$ is continuous, increasing and $r(0)=0$. Hence, it is one-to-one and its inverse $r^{-1}:[0,\infty)\to[0,\infty)$ is continuous and increasing. Moreover, since $G_0(0)=P_0E_0(0)_{|\mathrm{Im}P_0}\in\mathcal{B}(\mathrm{Im}P_0)$ by \eqref{l3.12.4} and \eqref{def-G-0}, from Lemma~\ref{l3.12}(iii) it follows that
\begin{equation}\label{r3.13.1}
\omega_0(G_0(0))=\sup\mathrm{Re}\sigma(G_0(0))\leq-q_1.
\end{equation}
By the definition of the growth rate of a semigroup one may define
$M_4:(0,1)\to[1,\infty)$ such that
\begin{equation}\label{r3.13.2}
M_4(\varkappa)=\sup_{t\geq 0}\big(e^{\frac{1+\varkappa}{2}q_1t}\|e^{tG_0(0)}\|\big)<\infty.
\end{equation}
\end{remark}
\begin{lemma}\label{l3.14}
Assume Hypotheses (H1)--(H5). Then the following estimates hold true, 
\begin{enumerate}
\item[(i)]
$\big\|e^{tK_\alpha}\big\|\leq M_4(\varkappa) e^{-\varkappa q(\alpha)t}$  for any $t\geq 0$, $\alpha\in[0,\eps_2(\varkappa)]$, $\varkappa\in(0,1)$, where $\eps_2:(0,1)\to (0,\infty)$ is the function defined by
\begin{equation}\label{def-eps-2}
\eps_2(\varkappa):=\min\Big\{\eps_1,q_2,\frac{(1-\varkappa)q_1\eps_1^2}{32M_2\nu},r^{-1}\Big(\frac{(1-\varkappa)q_1}{3072M_2^2}\Big)\Big\}>0.
\end{equation}
\item[(ii)]
$\big\|T_\alpha(t)\big\|\leq 3(8M_2+1)(M_3+M_4(\varkappa))e^{-\varkappa q(\alpha)t}$ for any $t\geq0$, $\alpha\in[0,\eps_3(\varkappa)]$, $\varkappa\in(0,1)$ where $\eps_3:(0,1)\to (0,\infty)$ is the function defined by 
\begin{equation}\label{def-eps-3}
\eps_3(\varkappa):=\min\Big\{\eps_2(\varkappa),\frac{\nu}{2q_1}\Big\}=\min\Big\{\eps_1,q_2,\frac{\nu}{2q_1},\frac{(1-\varkappa)q_1\eps_1^2}{32M_2\nu},r^{-1}\Big(\frac{(1-\varkappa)q_1}{3072M_2^2}\Big)\Big\}>0,
\end{equation}
and the constant $M_3$ and the function $M_4$ are defined in \eqref{M-3} and \eqref{r3.13.2}, respectively.
\end{enumerate}
\end{lemma}
\begin{proof}
(i) Fix $\varkappa\in (0,1)$. First, we note that the estimate	from Lemma~\ref{l3.12} (ii) is equivalent to
\begin{equation}\label{l3.14.0}
\|G_0(\alpha)-G_0(0)\|\leq \frac{8M_2\nu}{\eps_1^2}\alpha+768M_2^2 r(\alpha)\quad\mbox{for any}\quad\alpha\in[0,\eps_1].
\end{equation}
From Remark~\ref{r3.13}, \eqref{r3.13.2}, Gronwall's inequality and since
$\alpha\leq\frac{(1-\varkappa)q_1\eps_1^2}{32M_2\nu}$ and $r(\alpha)\leq\frac{(1-\varkappa)q_1}{3072M_2^2}$ for any $\alpha\in[0,\eps_2(\varkappa)]$
we infer that
\begin{equation}\label{l3.14.1}
\big\|e^{tG_0(\alpha)}\big\|\leq M_4(\varkappa)e^{\big(-\frac{1+\varkappa}{2}q_1+\frac{8M_2\nu}{\eps_1^2}\alpha+768M_2^2 r(\alpha)\big)t}\leq M_4(\varkappa)e^{\big(-\frac{1+\varkappa}{2}q_1+\frac{1-\varkappa}{2}q_1\big)t}= M_4(\varkappa)e^{-\varkappa q_1t}
\end{equation}
for any $t\geq 0$, $\alpha\in[0,\eps_2(\varkappa)]$. Since $K_\alpha=\alpha G_0(\alpha)$  by Lemma~\ref{l3.11} and $q(\alpha)=q_1\alpha$ for any $\alpha\in[0,\eps_2(\varkappa)]$ by Hypothesis (H3), using Lemma~\ref{l3.12}(i),  we conclude from \eqref{l3.14.1} that
\begin{equation}\label{l3.14.2}
\big\|e^{tK_\alpha}\big\|=\big\|e^{(t\alpha)G_0(\alpha)}\big\|\leq M_4(\varkappa)e^{-\varkappa q_1\alpha t}=M_4(\varkappa)e^{-\varkappa q(\alpha)t}\;\mbox{for any}\; t\geq 0,\;\alpha\in[0,\eps_2(\varkappa)].
\end{equation}

\noindent (ii) Fix again $\varkappa\in(0,1)$. From Lemma~\ref{l3.10} and  $\eps_3(\varkappa)\leq \min\{q_2,\frac{\nu}{2q_1}\}$, so that $q(\alpha)=q_1\alpha$ for any $\alpha\in[0,\eps_3(\varkappa)]$, we have  
\begin{equation}\label{l3.14-new}
\|\tS_\alpha(t)\|\leq M_3e^{-\frac{\nu}{2}t}=M_3e^{\big(-\frac{\nu}{2}+q_1\alpha\big)t}e^{-q(\alpha)t}\leq M_3e^{-q(\alpha)t}\;\mbox{for any}\; t\geq 0,\;\alpha\in[0,\eps_3(\varkappa)].
\end{equation}
From \eqref{new-l3.9.2}, \eqref{representation-S}, \eqref{l3.14-new} and (i) we conclude that  
\begin{align}\label{l3.15.12}
\|T_\alpha(t)\|&=\|U(\alpha)S_\alpha(t)U(\alpha)^{-1}\|\leq \|U(\alpha)\|\,\|S_\alpha(t)\|\,\|U(\alpha)^{-1}\|\leq 3\|S_\alpha(t)\|\nonumber\\
&\leq3\big(\big\|e^{tK_\alpha}\big\|\,\|P_0\|+\|\tS_\alpha(t)\|\,\|I_\bX-P_0\|\big)\leq 3(8M_2+1)(M_3+M_4(\varkappa))e^{-\varkappa q(\alpha)t}
\end{align}
for any $t\geq 0$, $\alpha\in[0,\eps_3(\varkappa)]$, proving the lemma.
\end{proof}
We conclude this subsection by proving one of our main results, the uniform in $\alpha$ exponential stability of the family of semigroups $\{T_\alpha(t)\}_{t\geq 0}$, $\alpha\geq 0$.
\begin{proof}[Proof of Theorem~\ref{t1.1}]
From Lemma~\ref{l3.2} and Lemma~\ref{l3.14}(ii) we derive
\begin{equation}\label{est-main}
\big\|T_\alpha(t)\big\|\leq \max\{\overline{M}(\varkappa,\eps_3(\varkappa)),3(8M_2+1)(M_3+M_4(\varkappa))\}e^{-\varkappa q(\alpha)t}
\end{equation}
for any $t\geq 0$, $\alpha\geq 0$, $\varkappa\in(0,1)$. Here $\overline{M}(\cdot,\cdot)$, $M_2$, $M_3$, $M_4(\cdot)$ and $\eps_3(\cdot)$ are defined in \eqref{overline-M}, \eqref{l3.5.4}, \eqref{M-3}, \eqref{r3.13.2} and \eqref{def-eps-3}, respectively. Tracing back the dependence of each of these quantities we conclude that estimate \eqref{t1.1.1} holds and that the function $M$ depends on the unperturbed operator $A$, $E_0(0)$ and the functions $\|E(\cdot)\|$, $q$ and $r$ and the relevant constants in Hypotheses (H3)-(H5).
\end{proof}	

\subsection{The special case when $0$ is a simple eigenvalue.}\label{sec3-3}

In this subsection we show that in the case when $0$ is a simple eigenvalue we can prove the main estimate \eqref{est-main} without using Hypothesis (H5). Throughout this subsection we assume  Hypotheses (H1), (H3), (H4) and Hypothesis (H2'). In this case $\mathrm{dim}\,\mathrm{Im}\,P_0=1$. This fact makes estimating $\|e^{tK_\alpha}\|$ significantly simpler. In particular the function $M_4$ in Lemma~\ref{l3.14} can be replaced by $1$, see \eqref{3.23.1} below.
\begin{lemma}\label{l3.23}
Assume Hypotheses (H1) (H2'), (H3) and (H4). Then,  
\begin{equation}\label{est-T-alpha-bis}
\big\|T_\alpha(t)\big\|\leq 3(M_3+1)(8M_2+1)e^{-q(\alpha)t}\quad\mbox{for any}\quad t\geq0, \alpha\in\big[0,\eps_4],
\end{equation}
where $\eps_4:=\min\big\{\eps_1,q_2,\frac{\nu}{2q_1}\big\}$, $M_3$ is taken from Lemma~\ref{l3.10}, $\eps_1$ from Lemma~\ref{new-l3.9}, $q_1$ and $q_2$ from Hypothesis (H3).
\end{lemma}
\begin{proof} 
Since $\dim\mathrm{Im}\,P_0=1$, there exists a function $k:[0,\eps_1]\to\CC$ such that 
\begin{equation}\label{l3.20.1}	
K_\alpha=k(\alpha) I_{\mathrm{Im}\,P_0}\quad\mbox{for any}\quad \alpha\in[0,\eps_1].
\end{equation}	
Since $K_\alpha=(B_\alpha)_{|\mathrm{Im}P_0}$, from \eqref{r3.7.1}, \eqref{def-tilde-B-K} and \eqref{l3.20.1} it follows that $k(\alpha)\in\sigma(B_\alpha)=\sigma(A_\alpha)$ for any $\alpha\in[0,\eps_1]$. From Hypothesis (H3) (i) we infer that 
\begin{equation}\label{3.22.10}
\mathrm{Re}\,k(\alpha)\leq-q(\alpha)\quad\mbox{for any}\quad\alpha\in[0,\eps_1].
\end{equation}	
From \eqref{l3.20.1} and \eqref{3.22.10} we immediately conclude that 
\begin{equation}\label{3.23.1}
\|e^{tK_\alpha}\|=e^{\mathrm{Re}\,k(\alpha)t}\leq e^{-tq(\alpha)}\quad\mbox{for any}\quad t\geq0, \alpha\in[0,\eps_1]. 
\end{equation}
Since $\eps_4\leq\min\big\{q_2,\frac{\nu}{2q_1}\big\}$, and so $q(\alpha)=q_1\alpha$ for any $\alpha\in[0,\eps_4]$, it follows from Lemma~\ref{l3.10} that
\begin{equation}\label{l3.23-new}
\|\tS_\alpha(t)\|\leq M_3e^{-\frac{\nu}{2}t}=M_3e^{\big(-\frac{\nu}{2}+q_1\alpha\big)t}e^{-q(\alpha)t}\leq M_3e^{-q(\alpha)t}\;\mbox{for any}\; t\geq 0,\;\alpha\in[0,\eps_4].
\end{equation}	
From Lemma~\ref{l3.10}, \eqref{representation-S} and \eqref{3.23.1} we infer  
\begin{align}\label{l3.15.124}
\|T_\alpha(t)\|&=\|U(\alpha)S_\alpha(t)U(\alpha)^{-1}\|\leq \|U(\alpha)\|\,\|S_\alpha(t)\|\,\|U(\alpha)^{-1}\|\leq 3\|S_\alpha(t)\|\nonumber\\
&\leq3\big(\big\|e^{tK_\alpha}\big\|\,\|P_0\|+\|\tS_\alpha(t)\|\,\|I_\bX-P_0\|\big)\leq 3(M_3+1)(8M_2+1)e^{-q(\alpha)t}
\end{align}
for any $t\geq 0$, $\alpha\in[0,\eps_1]$, proving the lemma.
\end{proof}
\begin{proof}[Proof of Theorem~\ref{t1.2}]
From  Lemma~\ref{l3.2} and Lemma~\ref{l3.23} it follows that 
\begin{equation}\label{est-main-bis}
\big\|T_\alpha(t)\big\|\leq \max\{\overline{M}(\varkappa,\eps_4),3(M_3+1)(8M_2+1)\}e^{-\varkappa q(\alpha)t}\;\mbox{for any}\; t\geq 0,\;\alpha\geq 0,\;\varkappa\in(0,1).
\end{equation}
Here $\overline{M}(\cdot)$, $M_2$, $M_3$, and $\eps_1$ are defined in \eqref{overline-M}, \eqref{l3.5.4}, \eqref{M-3} and Lemma~\ref{l3.23}, respectively. Arguing the same way as in the proof of Theorem~\ref{t1.1}, we conclude that estimate \eqref{t1.2.1} holds
and that the function $N$ depends on the unperturbed operator $A$ and the functions $\|E(\cdot)\|$ and $q$ and the relevant constants in Hypotheses (H3)-(H4).  
\end{proof}

\section{Applications to linear stability of planar traveling waves in reaction-diffusion systems}\label{s4}
In this section we give an application of our results to the case of families of analytic semigroups obtained by linearizing a reaction-diffusion system along a planar traveling wave (front). In particular, we give sufficient conditions for Lyapunov linear stability of such fronts, proving Proposition~\ref{p1.3}. First, we recall the reaction-diffusion system \eqref{RD-Sys} which reads as follows:
\begin{equation*}
u_t=D\Delta_xu+F(u),\; t\geq 0,\;x=(x_1,\dots,x_m)^{\mathrm{T}}\in\RR^m.
\end{equation*}
Here $F:\RR^k\to\RR^k$ is a function of class at least $\mathcal{C}^3$ and $D\in\CC^{k\times k}$ is a matrix satisfying the condition
\begin{equation}\label{RD-Sys-1}
\inf\mathrm{Re}\,\sigma(D)>0.   
\end{equation}
We recall that a planar traveling wave of \eqref{RD-Sys} is a solution of \eqref{RD-Sys} of the form $u(x,t)=\ohh(x_1-ct)$, where $c\in\RR$ and $\ohh:\RR\to\RR^k$ is a smooth function exponentially convergent at $\pm\infty$ to the limit values $\ohh_\pm$. One can readily check that the profile $\ohh$ satisfies the nonlinear system of equations
\begin{equation}\label{RD-Sys-TW}
D\ohh''+c\ohh'+F(\ohh)=0.   
\end{equation} 
Making the change of variables $y=x-ct\mathrm{\mathbf{e_1}}$, where $\mathrm{\mathbf{e_1}}=(1,0,\dots,0)^{\mathrm{T}}\in\RR^m$, we notice that equation \eqref{RD-Sys} is equivalent to
\begin{equation}\label{RD-Sys-2}
u_t=D\Delta_yu+c\partial_{y_1}u+F(u),\; t\geq 0,\;y=(y_1,\dots,y_m)\in\RR^m.
\end{equation}
We note that $\ohh$ is a standing wave solution of \eqref{RD-Sys-2} depending only on $y_1$. The linearization of \eqref{RD-Sys-2} along $\ohh$ reads as follows,   
\begin{equation}\label{RD-Sys-3}
u_t=\cL u,\;t\geq0,\;\quad\mbox{where}\quad \cL=D\Delta_y+cI_k\partial_{y_1}+\cM_{F'(\ohh)}.
\end{equation}
We recall that $\cM_{F'(\ohh)}$ denotes the operator of multiplication on $L^2(\RR^m,\CC^k)$ by the bounded, matrix valued function $F'(\ohh(y_1))$, while $\cL$ is considered as a closed, densely defined linear operator on $L^2(\RR^m,\CC^k)$ with domain $H^2(\RR^m,\CC^k)$. Moreover, as mentioned in the introduction, by taking Fourier transform in the variables $(y_2,\dots,y_m)\in\RR^{m-1}$, we infer that the linear operator $\cL$ is unitary equivalent to $\cM_{\hatt L}$, the operator of multiplication acting on $L^2\big(\RR^{m-1},L^2(\RR,\CC^k)\big)$ by the operator valued function     
\begin{equation}\label{hat-A}
\hatt L:\RR^{m-1}\to\mathcal{B}\big(H^2(\RR,\CC^k),L^2(\RR,\CC^k)\big),\;\hatt L(\xi)=D\partial_{y_1}^2+cI_k\partial_{y_1}+\cM_{V(\cdot,\xi)},
\end{equation} 
where $V:\RR^m\to\RR^k$ is defined by
\begin{equation}\label{def-V-0}
V(y_1,\xi)=F'(\ohh(y_1))-|\xi|^2D.
\end{equation}
For any $\xi\in\RR^{m-1}$ the linear operator $\hatt L(\xi)$ can be considered as a closed, densely defined linear operator on $L^2(\RR,\CC^k)$ with domain $H^2(\RR,\CC^k)$. 

It is well-known that elliptic operators generate analytic semigroups, see, e.g.,  \cite{Agmon,ABHN,Henri,lunardi,P,Stewart1, Stewart2}. In the next lemma we show that condition \eqref{RD-Sys-1} is enough to infer the analyticity of the semigroup generated by $\hatt L(0)$.  
\begin{lemma}\label{l4.1} If $\inf\mathrm{Re}\,\sigma(D)>0$ then the linear operator $\hatt L(\xi)$, defined in \eqref{hat-A}, is sectorial, for any $\xi\in\RR^{m-1}$, hence it generates an analytic semigroup. In particular, $\hatt L(0)$ satisfies Hypothesis (H1).
\end{lemma}
\begin{proof}
Let $d_0=\inf\mathrm{Re}\,\sigma(D)>0$. First, we show that the linear operator $D\partial_{y_1}^2$ is sectorial. Denoting by $\mathcal{F}_1$	the Fourier Transform with respect to the variable $y_1\in\RR$, one can readily check that 
\begin{equation}\label{l4.1-new}
D\partial_{y_1}^2=\mathcal{F}_1^{-1}M_{\hatt D}\mathcal{F}_1,\;\mbox{where}\;\hatt D:\RR\to\CC^{k\times k}\;\mbox{is defined by}\; \hatt D(\xi_1)=-\xi_1^2D.
\end{equation}
Since $D\in\CC^{k\times k}$ we have  
\begin{equation}\label{l4.1.1}
\omega_0(-D)=\sup\mathrm{Re}\,\sigma(-D)=-\inf\mathrm{Re}\,\sigma(D)=-d_0<0.
\end{equation}
From the definition of the growth rate of a semigroup we obtain that there exists $C_0>0$ such that 
\begin{equation}\label{l4.1.2}
\|e^{-tD}\|\leq C_0e^{-\frac{d_0}{2}t}\quad\mbox{for any}\quad t\geq0.
\end{equation}
Since $\{\lambda\in\CC:\mathrm{Re}\,\lambda\geq-\frac{d_0}{2}\}\subset\rho(-D)$ we have 
\begin{equation}\label{l4.1.3}
\lambda(\lambda I_k+D)^{-1}=I_k-D(\lambda I_k+D)^{-1}=I_k-D\int_0^\infty e^{-\lambda t}e^{-tD}\rmd t\quad\mbox{whenever}\quad \mathrm{Re}\,\lambda>-\frac{d_0}{2}. 
\end{equation}
Therefore, from \eqref{l4.1.2} and \eqref{l4.1.3} we immediately obtain   
\begin{equation}\label{l4.1.4}
\sup_{\mathrm{Re}\,\lambda\geq0}\|\lambda(\lambda I_k+D)^{-1}\|\leq1+\frac{2\|D\|C_0}{d_0}<\infty.
\end{equation}
From Lemma~\ref{lemma-Lunardi} it follows that there exists $\tC_0\geq 1$ and $\varphi_0\in(\frac{\pi}{2},\pi)$ such that 
\begin{equation}\label{l4.1.5}
\Omega_{0,\varphi_0}\subset\rho(-D),\quad\|(\lambda I_k+D)^{-1}\|\leq\frac{\tC_0}{|\lambda|}\quad\mbox{for any}\quad\lambda\in\Omega_{0,\varphi_0}.
\end{equation}
Since $\frac{\lambda}{\xi_1^2}\in\Omega_{0,\varphi_0}$ for any $\lambda\in\Omega_{0,\varphi_0}$ and $\xi_1\in\RR\setminus\{0\}$, from \eqref{l4.1-new} and \eqref{l4.1.5} we conclude that 
\begin{equation}\label{l4.1.6}
\Omega_{0,\varphi_0}\subset\rho\big(\hatt D(\xi_1)\big),\quad\big\|\big(\lambda I_k-\hatt D(\xi_1)\big)^{-1}\big\|\leq\frac{\tC_0}{|\lambda|}\quad\mbox{for any}\quad\lambda\in\Omega_{0,\varphi_0},\,\xi_1\in\RR, 
\end{equation}
which is equivalent to 
\begin{equation}\label{l4.1.7}
\Omega_{0,\varphi_0}\subset\rho(M_{\hatt D}),\quad\|R(\lambda, M_{\hatt D})\|\leq\frac{\tC_0}{|\lambda|}\quad\mbox{for any}\quad\lambda\in\Omega_{0,\varphi_0}.
\end{equation}
From \eqref{l4.1.7} we infer that the linear operator $M_{\hatt D}$, and hence $D\partial_{y_1}^2$, is a sectorial operator. 
Since $D\in\CC^{k\times k}$ is an invertible matrix one can readily check that 
\begin{equation}\label{l4.1.8}
\|\partial_{y_1}u\|_2^2=-\langle \partial_{y_1}^2u,u\rangle_{L^2(\RR,\CC^k)}=-\langle D\partial_{y_1}^2u,(D^{-1})^*u\rangle_{L^2(\RR,\CC^k)}\leq\|D^{-1}\|\,\|D\partial_{y_1}^2u\|_2\,\|u\|_2\
\end{equation}
for any $u\in H^2(\RR,\CC^k)$, which implies that 
\begin{equation}\label{l4.1.9}
\|\partial_{y_1}u\|_2\leq\|D^{-1}\|^{1/2}\,\|D\partial_{y_1}^2u\|_2^{1/2}\,\|u\|_2^{1/2}\leq\|D^{-1}\|^{1/2}\,\|u\|_{\dom(D\partial_{y_1}^2)}^{1/2}\,\|u\|_2^{1/2}
\end{equation}
for any $u\in\dom(D\partial_{y_1}^2)=H^2(\RR,\CC^k)$. Applying the results from \cite[Chapter 2]{lunardi}, we infer that $D\partial_{y_1}^2+cI_k\partial_{y_1}$ is also a sectorial operator. Since $\ohh$, and thus $V$, are bounded functions, from Lemma~\ref{l2.10} and \eqref{hat-A} we conclude that $\hatt L(\xi)$ is a sectorial operator for any $\xi\in\RR^{m-1}$, proving the lemma. 
\end{proof}
Next, we note that the operator valued function $\hatt L$ has the representation
\begin{equation}\label{hat-A-representation}
\hatt L(\xi)=\hatt L(0)+E_{\mathrm{RD}}(|\xi|^2)\;\mbox{for any}\;\xi\in\RR^{m-1},
\end{equation}
where $E_{\mathrm{RD}}:[0,\infty)\to\mathcal{B}\big(L^2(\RR,\CC^k)\big)$, is defined by $E_{\mathrm{RD}}(\alpha)=-\alpha \cM_D$. 

\begin{remark}\label{r4.2} Assuming Hypothesis (RD) from the Introduction and \eqref{RD-Sys-1}, from Lemma~\ref{l4.1} we derive that $\hatt L(0)$ satisfies Hypotheses (H1)-(H2). Hence, it follows from Lemma~\ref{l3.5} that the semigroup generated by $\hatt L(0)$ has a block representation of the form \eqref{representation-A-T} and satisfies the estimate \eqref{est-sem-tilde-A}. We infer that the semigroup generated by $\hatt L(0)$ is bounded, that is $\tM_0=\sup_{t\geq 0}\|e^{t\hatt L(0)}\|<\infty$, which implies that 
\begin{equation}\label{r4.2.1}
\|R(\lambda,\hatt L(0))\|\leq \frac{\tM_0}{\mathrm{Re}\,\lambda}\;\mbox{whenver}\;\mathrm{Re}\,\lambda>0.
\end{equation}	 
\end{remark}
In many applications, the diffusion rates of various components of the vector valued function $u$ in \eqref{RD-Sys} are close to each other. In this case we can prove the Lyapunov linear stability of the front $\ohh$ using the results of Section~\ref{s3}. First, we prove the following lemma.
\begin{lemma}\label{l4.3}
Assume Hypothesis (RD) and that the matrix $D\in\CC^{k\times k}$ is sufficiently close to a diagonal matrix in the sense that there exists some $d>0$ such that for $\tM_0$ from \eqref{r4.2.1} one has,
\begin{equation}\label{D-condition}
\|D-dI_k\|<\frac{d}{\tM_0}.
\end{equation}
Then, the family of operators $A_{\mathrm{RD}}(\alpha):=\hatt L(0)-\alpha \cM_{D}$, $\alpha\geq0$, satisfies Hypotheses (H1)-(H5). 	
\end{lemma}

\begin{proof}
Denoting by $J=D-dI_k\in\CC^{k\times k}$ we have $\|J\|<\frac{d}{\tM_0}$. One can readily check that $\sigma(J)\subset D(0,\frac{d}{\tM_0})$. Since $\tM_0=\sup_{t\geq 0}\|e^{t\hatt L(0)}\|\geq 1$, we have  $\inf\mathrm{Re}\,\sigma(J)>-d$, which implies that $\inf\mathrm{Re}\,\sigma(D)>0$. From Remark~\ref{r4.2} we conclude that the linear operator $\hatt L(0)$ satisfies Hypotheses (H1) and (H2). 

Next, we study the spectrum of $A_{\mathrm{RD}}(\alpha)$. Since $\frac{\|J\|\tM_0}{d}<1$, there exists $\delta_0\in(0,1)$ such that $\frac{\|J\|\tM_0}{d}<(1-\delta_0)^2$. Next, we fix $\alpha\geq0$ and $\lambda\in\CC$ with $\mathrm{Re}\lambda>-\delta_0\alpha d$. Then, from Hypothesis (RD) we note that $\lambda+\alpha d\in\rho\big(\hatt L(0)\big)$. Moreover,  
since $D=dI_k+J$ we have  
\begin{align}\label{l4.3.1}
&\lambda I_{L^2(\RR,\CC^k)}-A_{\mathrm{RD}}(\alpha)=(\lambda+\alpha d)I_{L^2(\RR,\CC^k)}-\hatt L(0)+ \alpha \cM_{J}\nonumber\\&\qquad=\Big(I_{L^2(\RR,\CC^k)}+\alpha \cM_{J}R\big(\lambda+\alpha d,\hatt L(0)\big)\Big)\big((\lambda+\alpha d) I_{L^2(\RR,\CC^k)}-\hatt L(0)\big).
\end{align}
From \eqref{r4.2.1} we obtain  
\begin{equation}\label{l4.3.2}
\big\|\alpha \cM_{J}R\big(\lambda+\alpha d,\hatt L(0)\big)\big\|\leq\frac{\alpha\|J\|\tM_0}{\mathrm{Re}\lambda+\alpha d}\leq\frac{\|J\|\tM_0}{(1-\delta_0) d}\leq1-\delta_0<1,
\end{equation}
which implies that $I_{L^2(\RR,\CC^k)}+\alpha \cM_{J}R\big(\lambda+\alpha d,\hatt L(0)\big)$ is invertible and 
\begin{equation}\label{l4.3.3}
\Big\|\Big(I_{L^2(\RR,\CC^k)}+\alpha \cM_{J}R\big(\lambda+\alpha d,\hatt L(0)\big)\Big)^{-1}\Big\|\leq\frac{1}{\delta_0}.
\end{equation}
From \eqref{l4.3.1} and \eqref{l4.3.3} we infer that $\lambda\in\rho\big(A_{\mathrm{RD}}(\alpha)\big)$ and
\begin{align}\label{l4.3.4}
\big\|R\big(\lambda,A_{\mathrm{RD}}(\alpha)\big)\big\|&\leq\big\|R\big(\lambda+\alpha d,\hatt L(0)\big)\big\|\,\Big\|\Big(I_{L^2(\RR,\CC^k)}+\alpha \cM_{J}R\big(\lambda+\alpha d,\hatt L(0)\big)\Big)^{-1}\Big\|\nonumber\\
&\leq\frac{\tM_0}{\delta_0(\mathrm{Re}\lambda+\alpha d)}\leq\frac{\tM_0}{\delta_0(\mathrm{Re}\lambda+\delta_0\alpha d)}.
\end{align}
We conclude that  
\begin{equation}\label{l4.3.5}
\sigma\big(A_{\mathrm{RD}}(\alpha)\big)\subseteq\{\lambda\in\CC:\mathrm{Re}\lambda\leq-q(\alpha)\}\;\mbox{and}\;\big\|R\big(\lambda,A_{\mathrm{RD}}(\alpha)\big)\big\|\leq\frac{\tM_0}{\delta_0(\mathrm{Re}\lambda+q(\alpha))}
\end{equation}
whenever $\mathrm{Re}\lambda>-q(\alpha)$, where the function $q:[0,\infty)\to[0,\infty)$ is defined by $q(\alpha)=\delta_0\alpha d$. Since $\lim_{\alpha\to\infty}\frac{\|E_{\mathrm{RD}}(\alpha)\|}{q(\alpha)}=\frac{\|D\|}{\delta_0d}$, Hypothesis (H3) is satisfied. From \eqref{hat-A-representation} one can readily check that Hypotheses (H4) and (H5) are satisfied. 
\end{proof}
We are now ready to prove the main result of this section, Proposition~\ref{p1.3}, assuming that $D\in\CC^{k\times k}$ is sufficiently close to $dI_k$ in the sense that $\|D-dI_k\|<\frac{d}{\tM_0}$

\begin{proof}[Proof of Proposition~\ref{p1.3}] First, we recall that $\tM_0=\sup_{t\geq0}\|e^{t\hatt L(0)}\|$.
From Lemma~\ref{l4.3} and Theorem~\ref{t1.1} it follows that the family of semigroups generated by $A_{\mathrm{RD}}(\alpha)$, $\alpha\geq 0$, is uniformly stable. Setting $\varkappa=\frac{1}{2}$ in Theorem~\ref{t1.1} we obtain that there exists a constant $\hatt M_0>0$ such that 
\begin{equation}\label{t4.4-est}
\|e^{tA_{\mathrm{RD}}(\alpha)}\|\leq\hatt M_0e^{-\frac{\delta_0\alpha d}{2}t}\;\mbox{for any}\;t\geq0,\, \alpha\geq 0.
\end{equation}
Using the identity $\hatt L(\xi)=A_{\mathrm{RD}}(|\xi^2|)$ for any $\xi\in\RR^{m-1}$, from \eqref{t4.4-est} we infer that
\begin{equation}\label{t4.4-est-bis}
\|e^{t\hatt L(\xi)}\|\leq\hatt M_0e^{-\frac{\delta_0|\xi|^2d}{2}t}\;\mbox{for any}\;t\geq0,\,\xi\in\RR^{m-1}.
\end{equation}
Since $\cL$, the linearization along the front $\ohh$, is unitary equivalent to $\cM_{\hatt L}$, from \eqref{t4.4-est-bis} we see that the semigroup generated by $\cL$ is bounded, proving that the planar front $\ohh$ is Lyapunov linearly stable.
\end{proof}

\section{Applications to linear stability of planar fronts in the bidomain equation }\label{s5}
In this section we give yet another application of our results to the Lypunov linear stability of planar traveling waves (fronts) in the bidomain Allen-Cahn model \eqref{bidomain}. In \cite{MaMo} the authors show that the model has \textit{spectrally stable} planar traveling waves (fronts). Our goal is to show that the same condition also guarantees the Lyapunov linear stability of such fronts. Following \cite{MaMo} we consider equation \eqref{bidomain} on $\RR^2$. It is well-known that \eqref{bidomain} has planar traveling waves, that is, solutions of the form
\begin{equation}\label{bidomain-TW}
(u,u_i,u_e)(x,t)=(\ow,\ow_i,\ow_e)(x_1\cos\gamma+x_2\sin\gamma-ct), \;x=(x_1,x_2)\in\RR^2, t\geq 0,
\end{equation}
for some $c,\gamma\in\RR$. Moreover, the profile $\ow$ can be chosen such that 
\begin{equation}\label{bidomain-TW-2}
\ow\;\mbox{is decreasing},\;\lim_{s\to-\infty}e^{-\rho_0 s}(\ow(s)-1)=\lim_{s\to\infty}e^{\rho_0 s}\ow(s)=0,\;\mbox{for some}\; \rho_0>0.
\end{equation}
Making in \eqref{bidomain} the change of variables
\begin{equation}\label{var-change-bidomain}
y=\left(\begin{matrix}y_1\\
y_2
\end{matrix}\right):=R_{-\gamma}\left(\begin{matrix}x_1\\
x_2
\end{matrix}\right)-\left(\begin{matrix}ct\\
0
\end{matrix}\right),\quad\mbox{where}\quad R_\gamma=\begin{bmatrix} \cos\gamma  &
-\sin\gamma\\
\sin\gamma& \cos\gamma\end{bmatrix},
\end{equation}
we obtain the system
\begin{equation}\label{bidomain-var-change}
\left\{\begin{array}{ll}
u_t=\nabla_y\cdot(A_{i,\gamma}\nabla_y u_i)+c\partial_{y_1}u+f(u),\\
\nabla_y\cdot(A_{i,\gamma}\nabla_y u_i+A_{e,\gamma}\nabla_y u_e)=0,\\
u=u_i-u_e.\end{array}\right.t\geq0,\; y\in\RR^2.
\end{equation} 
Here, the symmetric, positive definite matrices $A_{i,\gamma}$ and $A_{e,\gamma}$ are defined by
\begin{equation}\label{A-i-e-gamma}
A_{i,\gamma}=R_\gamma A_i R_{-\gamma},\quad A_{e,\gamma}=R_\gamma A_e R_{-\gamma}.
\end{equation}
In the new coordinate system the front $(\ow,\ow_i,\ow_e)$ travels along the $y_1$-axis and it is a standing wave solution of \eqref{bidomain-var-change} depending only on $y_1$. In addition, the second and third equation of \eqref{bidomain-var-change} are linear equations. Therefore, the linearization  of \eqref{bidomain-var-change} along the front $(\ow,\ow_i,\ow_e)$ is given by 
\begin{equation}\label{bidomain-linear}
\left\{\begin{array}{ll}
v_t=\nabla_y\cdot(A_{i,\gamma}\nabla_y v_i)+c\partial_{y_1}v+f'(\ow(y_1))v,\\
\nabla_y\cdot(A_{i,\gamma}\nabla_y v_i+A_{e,\gamma}\nabla_y v_e)=0,\\
v=v_i-v_e.\end{array}\right.t\geq0,\; y\in\RR^2
\end{equation} 
Following \cite{MaMo}, we consider this system in $L^2(\RR^2,\CC^3)$. Taking Fourier Transform in $y=(y_1,y_2)$, denoted by $\cF$, we can eliminate the variables $v_i$ and $v_e$ from \eqref{bidomain-linear} to obtain the following equation
\begin{equation}\label{bidomain-linear-reduced}
v_t=\cA v,\;t\geq0,\;\quad\mbox{where}\quad \cA=-\cL_\gamma +c\partial_{y_1}+\cM_{f'(\ow)}.
\end{equation}
The linear operator $\cL_\gamma:H^2(\RR^2)\to L^2(\RR^2)$ is defined as the Fourier multiplier 
\begin{equation}\label{def-CL-gamma}
\cL_\gamma=\cF^{-1}\cM_{Q_\gamma}\cF,\;Q_\gamma(\xi)=\frac{Q_{i,\gamma}(\xi)Q_{e,\gamma}(\xi)}{Q_{i,\gamma}(\xi)+Q_{e,\gamma}(\xi)},\;Q_{i/e,\gamma}(\xi)=\xi^{\mathrm{T}}\cdot A_{i/e,\gamma}\xi,\;\xi\in\RR^2. 
\end{equation}
We recall that the wave $\ow$ is called Lyapunov linearly stable if the semigroup generated by $\cA$ is bounded.

Taking Fourier Transform with respect to $y_2\in\RR$, denoted by $\cF_2$, we note that the linear operator $\cA$ is unitary equivalent to $\cM_{\hatt A}$, the operator of multiplication on $L^2\big(\RR,L^2(\RR)\big)$ by the operator valued function     
\begin{equation}\label{def-A-hat}
\hatt A:\RR\to\mathcal{B}\big(H^2(\RR),L^2(\RR)\big),\;\hatt A(\xi_2)=-\cF_1^{-1}\cM_{Q_\gamma(\cdot,\xi_2)}\cF_1+c\partial_{y_1}+\cM_{f'(\ow)}.
\end{equation} 
Here $\cF_1$ denotes the Fourier transform with respect to the variable $y_1\in\RR$. For more details we refer to \cite[Section 2]{MaMo}.

To study the Lyapunov linear stability of the front $\ow$ we study the family of semigroups generated by $\hatt A(\xi_2)$, $\xi_2\in\RR$. We recall that the function $Q_\gamma$ has the representation:
\begin{equation*}
Q_\gamma(\xi_1,\xi_2)=\left\{\begin{array}{ll} \xi_2^2\Big(p\Big(\frac{\xi_1}{\xi_2}\Big)+g\Big(\frac{\xi_1}{\xi_2}\Big)\Big),& \xi_1\in\RR,\,\xi_2\in\RR\setminus\{0\},\\
N_0^2\xi_1^2,& \xi_1\in\RR,\,\xi_2=0, \end{array}\right.\;\mbox{where},
\end{equation*}
\begin{equation*}
p(s)=N_0^2(s-\eta_1)^2+\eta_0,\quad g(s)=\frac{\beta_1s+\beta_0}{s^2+1}.
\end{equation*}
as pointed out in \eqref{rep-Q-gamma} and \eqref{def-p-g}. Also, we recall that the constants $N_0$, $\beta_0$, $\beta_1$, $\eta_0$ and $\eta_1$ depend on $\nu_1$, $\nu_2$ and $\gamma$ only. For the exact formulas we refer to \cite[Formula (2.20)]{MaMo}. From \eqref{rep-Q-gamma} one can readily check that 
\begin{equation}\label{hat-A-bidomain}
\hatt A(0)=N_0^2\partial _{y_1}^2+c\partial_{y_1}+\cM_{f'(\ow)}.
\end{equation}
The linear operator $\hatt A(0)$ is a second order differential operator, therefore we can determine the properties of its spectrum, see for instance \cite{KP,RS2,S,Volpert3}. In the next lemma we summarize the most relevant properties of the spectrum of $\hatt A(0)$ given in \cite{MaMo}.
\begin{lemma}\cite[Proposition 3.1]{MaMo}\label{l5.1}
The operator $\hatt A(0)$ satisfies the following properties:
\begin{enumerate}
\item[(i)] There exists $\nu>0$ such that $\sup\,\mathrm{Re}\big(\sigma(\hatt A(0))\setminus\{0\}\big)\leq-\nu$;
\item[(ii)] $0$ is a simple eigenvalue of $\hatt A(0)$. Moreover, $\mathrm{Ker}\,\hatt A(0)=\mathrm{Span}\{\ow'\}$;
\item[(iii)] There exists $M_\rmb>0$ such that 
\begin{equation*}
\big\|R\big(\lambda,\hatt A(0)\big)\big\|\leq\frac{M_\rmb}{|\lambda|}\;\mbox{whenever}\;\mathrm{Re}\,\lambda\geq 0, \lambda\ne 0.
\end{equation*}	
\end{enumerate}
\end{lemma}
\begin{remark}\label{r5.2}
From Lemma~\ref{lemma-Lunardi} and Lemma~\ref{l5.1} (ii) and (iii), we infer that $\hatt A(0)$ is a sectorial operator, hence $\hatt A(0)$ satisfies Hypothesis (H1). Moreover, from Lemma~\ref{l5.1} (i) and (ii) it follows that $\hatt A(0)$ satisfies Hypothesis (H2').	
\end{remark}
We recall the following notation needed to formulate \eqref{spectral-stable}, a sufficient condition for spectral stability of the planar front $\ow$,
\begin{equation}\label{g-bar-Delta}
g_{\mathrm{inf}}=\inf_{s\in\RR}g(s),\;g_{\mathrm{sup}}=\sup_{s\in\RR}g(s),\;\og=\frac{g_{\mathrm{inf}}+g_{\mathrm{sup}}}{2},\;g_\Delta=\frac{g_{\mathrm{sup}}-g_{\mathrm{inf}}}{2}.
\end{equation}
A crucial ingredient in the proof of \eqref{spectral-stable} is the identity 
\begin{equation}\label{hat-A-0-xi}
\hatt A(\xi_2)=\cM_{e^{\rmi\eta_1\xi_2\cdot}}\big(\hatt A(0)+H(\xi_2)\big)\cM_{e^{-\rmi\eta_1\xi_2\cdot}}\;\mbox{for any}\;\xi_2\in\RR,
\end{equation}
where $H:\RR\to\mathcal{B}\big(L^2(\RR)\big)$ is defined by
\begin{equation}\label{def-H}
H(\xi_2)=\left\{\begin{array}{ll} -\xi_2^2\cF_1^{-1}\cM_{g(\frac{\cdot}{\xi_2}+\eta_1)}\cF_1-(\eta_0\xi_2^2-c\rmi\eta_1\xi_2)I_{L^2(\RR)},& \xi_2\in\RR\setminus\{0\},\\
0,& \xi_2=0. \end{array}\right.
\end{equation}
\begin{remark}\label{r5.3}
From \cite[Lemma 4.1]{MaMo} we see that the operator valued function $H$ is continuous on $\RR$ in $\mathcal{B}\big(L^2(\RR)\big)$-norm. Moreover, assuming that $\eta_0>M_\rmb g_\Delta-\og$ and using the same argument as in \cite[Theorem 3.2]{MaMo} we can prove that 
\begin{equation}\label{spectrum-A-H}
\sigma\big(\hatt A(0)+H(\xi_2)\big)\subseteq\{\lambda\in\CC:\mathrm{Re}\lambda\leq-(\eta_0-M_\rmb g_\Delta+\og)\xi_2^2\}\;\mbox{for any}\;\xi_2\in\RR;
\end{equation}
\begin{equation}\label{resolvent-A-H}
\|R\big(\hatt A(0)+H(\xi_2)\big)\|\leq\frac{M_\rmb}{\mathrm{Re}\lambda+(\eta_0-M_\rmb g_\Delta+\og)\xi_2^2}\;\mbox{whenever}\;\mathrm{Re}\lambda>-(\eta_0-M_\rmb g_\Delta+\og)\xi_2^2.
\end{equation}
To check the two properties above we first define $\tH:\RR\to\mathcal{B}\big(L^2(\RR)\big)$ by
\begin{equation}\label{def-tilde-H} \tH(\xi_2):=H(\xi_2)+\Big((\eta_0+\og)\xi_2^2-c\rmi\eta_1\xi_2\Big)I_{L^2(\RR)}.
\end{equation}	
We note that \eqref{spectrum-A-H} and \eqref{resolvent-A-H} hold true if $\xi_2=0$. Next, we fix $\xi_2\in\RR\setminus\{0\}$ and $\lambda\in\CC$ such that $\mathrm{Re}\lambda>-(\eta_0-M_\rmb g_\Delta+\og)\xi_2^2$ and we set $\tlambda:=\lambda+(\eta_0+\og)\xi_2^2-c\rmi\eta_1\xi_2$. Using \eqref{g-bar-Delta} one can readily check that 
\begin{equation}\label{r5.3.1}
\mathrm{Re}\tlambda=\mathrm{Re}\lambda+(\eta_0+\og)\xi_2^2>M_\rmb g_\Delta\xi_2^2>0\;\mbox{and}\;\|\tH(\xi_2)\|\leq\xi_2^2\Big\|\og-g\big(\frac{\cdot}{\xi_2}+\eta_1\big)\Big\|_\infty\leq\xi_2^2 g_\Delta.
\end{equation}
From \eqref{r5.3.1} and Lemma~\ref{l5.1} we have $\tlambda\in\rho\big(\hatt A(0)\big)$ and $\big\|R\big(\tlambda,\hatt A(0)\big)\big\|\leq\frac{M_\rmb}{|\tlambda|}\leq\frac{M_\rmb}{\mathrm{Re}\tlambda}$, which implies 
\begin{equation}\label{r5.3.2}
\big\|\tH(\xi_2)R\big(\tlambda,\hatt A(0)\big)\big\|\leq\|\tH(\xi_2)\|\,\big\|R\big(\tlambda,\hatt A(0)\big)\big\|\leq\xi_2^2 g_\Delta\frac{M_\rmb}{\mathrm{Re}\tlambda}<1.
\end{equation}
Using elementary spectral theory, from \eqref{r5.3.2} we obtain that $\tlambda\in\rho\big(\hatt A(0)+\tH(\xi_2)\big)$, hence $\lambda\in\rho\big(\hatt A(0)+H(\xi_2)\big)$. Moreover, 
\begin{align}\label{r5.3.3}
\big\|R\big(\lambda,\hatt A(0)+H(\xi_2)\big)\big\|&=\big\|R\big(\tlambda,\hatt A(0)+\tH(\xi_2)\big)\big\|\leq\frac{\big\|R\big(\tlambda,\hatt A(0)\big)\big\|}{1-\big\|\tH(\xi_2)R\big(\tlambda,\hatt A(0)\big)\big\|}\nonumber\\
&\leq\frac{M_\rmb}{\mathrm{Re}\tlambda-M_\rmb g_\Delta\xi_2^2}=\frac{M_\rmb}{\mathrm{Re}\lambda+(\eta_0-M_\rmb g_\Delta+\og)\xi_2^2},
\end{align}
proving \eqref{spectrum-A-H} and \eqref{resolvent-A-H}.
\end{remark}	
To prove uniform estimates for the family of semigroups generated by $\hatt A(\xi_2)$, $\xi_2\in\RR$, we introduce the operator valued functions $E_\pm:[0,\infty)\to\mathcal{B}\big(L^2(\RR)\big)$ defined by
\begin{equation}\label{def-E-pm} E_\pm(\alpha):=H(\pm\sqrt{\alpha})\mp c\rmi\eta_1\sqrt{\alpha}I_{L^2(\RR)}.
\end{equation}
\begin{lemma}\label{l5.4} Assume that $\eta_0>M_\rmb g_\Delta-\og$. Then, the family of operators $\tA_\pm(\alpha):=\hatt A(0)+E_\pm(\alpha)$, $\alpha\geq 0$, satisfies Hypotheses (H3) and (H4).
\end{lemma}	
\begin{proof} 
From Remark~\ref{r5.3} and \eqref{def-E-pm} we immediately conclude that the operator valued function $E_\pm$ is continuous on $[0,\infty)$ in $\mathcal{B}\big(L^2(\RR)\big)$-norm. It follows from \eqref{def-E-pm} that 
\begin{align}\label{l5.4.1}
\sigma\big(\hatt A(0)+E_\pm(\alpha)\big)&=\sigma\big(\hatt A(0)+H(\pm\sqrt{\alpha})\big)\mp c\rmi\eta_1\sqrt{\alpha},\nonumber\\R\big(\lambda,\hatt A(0)+E_\pm(\alpha)\big)&=R\big(\lambda\pm \rmi\eta_1\sqrt{\alpha},\hatt A(0)+H(\pm\sqrt{\alpha})\big)
\end{align}
for any $\alpha\geq 0$ and any $\lambda\in\rho\big(\hatt A(0)+E_\pm(\alpha)\big)$. From \eqref{spectrum-A-H}, \eqref{resolvent-A-H} and \eqref{l5.4.1} we infer that 
\begin{equation}\label{l5.4.2}
\sigma\big(\hatt A(0)+E_\pm(\alpha)\big)\subseteq\{\lambda\in\CC:\mathrm{Re}\lambda\leq-(\eta_0-M_\rmb g_\Delta+\og)\alpha\}\;\mbox{for any}\;\alpha\geq 0;
\end{equation}
\begin{equation}\label{l5.4.3}
\|R\big(\hatt A(0)+E_\pm(\alpha)\big)\|\leq\frac{M_\rmb}{\mathrm{Re}\lambda+(\eta_0-M_\rmb g_\Delta+\og)\alpha}\;\mbox{whenever}\;\mathrm{Re}\lambda>-(\eta_0-M_\rmb g_\Delta+\og)\alpha.
\end{equation}
Let $q:[0,\infty)\to[0,\infty)$ be the function defined by $q(\alpha)=(\eta_0-M_\rmb g_\Delta+\og)\alpha$. From \eqref{l5.4.2} and \eqref{l5.4.3} we see that conditions (i) and (ii) of Hypothesis (H3) are satisfied. Since $g\in L^\infty(\RR)$, by \eqref{def-p-g}, from \eqref{def-H} and \eqref{def-E-pm} we obtain
\begin{equation}\label{l5.4.4}
\|E_\pm(\alpha)\|\leq \alpha(\|g\|_\infty+\eta_0)\;\mbox{for any}\;\alpha\geq 0,
\end{equation}
which implies that $\limsup\limits_{\alpha\to\infty}\frac{\|E_\pm(\alpha)\|}{q(\alpha)}\leq\|g\|_\infty+\eta_0<\infty$. Hence, condition (iii) of Hypothesis (H3) is satisfied. Finally,  it follows from \eqref{l5.4.4} that Hypothesis (H4) is satisfied, proving the lemma.  
\end{proof}
We are now ready to prove the main result of this section, the Lyapunov linear stability of the planar front $\ow$.
\begin{proof}[Proof of Proposition~\ref{p1.4}] From Remark~\ref{r5.2} and Lemma~\ref{l5.4} the family of operators $\tA_\pm(\alpha)=\hatt A(0)+E_\pm(\alpha)$, $\alpha\geq 0$, satisfies Hypotheses (H1), (H2'), (H3) and (H4). Hence, from Theorem~\ref{t1.2} we conclude that the family of analytic semigroups generated by $\tA_\pm(\alpha)$ is uniformly exponentially stable for $\alpha\geq0$. Setting $\varkappa=\frac{1}{2}$ in Theorem~\ref{t1.2} we see that there exists a constant $\tM_\rmb>0$ such that 
\begin{equation}\label{t5.5.1}
\|e^{t\tA_\pm(\alpha)}\|\leq \tM_\rmb e^{-\frac{(\eta_0-M_\rmb g_\Delta+\og)\alpha}{2}t}\;\mbox{for any}\;t\geq0,\, \alpha\geq0.
\end{equation}
From \eqref{def-E-pm} we have  
\begin{equation}\label{t5.5.2}
H(\xi_2)=E_\sigma(\xi_2^2)+c\rmi\eta_1\xi_2I_{L^2(\RR)}\;\mbox{for any}\;\xi_2\in\RR,\;\mbox{where}\;\sigma=\left\{\begin{array}{ll} +,& \xi_2\geq0,\\
-,& \xi_2<0. \end{array}\right.
\end{equation}
From \eqref{hat-A-0-xi} and \eqref{t5.5.2} one can readily check that 
\begin{equation}\label{t5.5.3}
e^{t\hatt A(\xi_2)}=\cM_{e^{\rmi\eta_1\xi_2\cdot}}\,e^{t\big(\hatt A(0)+H(\xi_2)\big)}\cM_{e^{-\rmi\eta_1\xi_2\cdot}}=e^{c\rmi\eta_1\xi_2t}\cM_{e^{\rmi\eta_1\xi_2\cdot}}\,e^{t\tA_\sigma(\xi_2^2)}\cM_{e^{-\rmi\eta_1\xi_2\cdot}}
\end{equation}	
for any $t\geq0$ and $\xi_2\in\RR$.	From \eqref{t5.5.1} and \eqref{t5.5.3} we conclude that 
\begin{equation}\label{t5.5.4}
\|e^{t\hatt A(\xi_2)}\|\leq \tM_\rmb e^{-\frac{(\eta_0-M_\rmb g_\Delta+\og)\xi_2^2}{2}t}\;\mbox{for any}\;t\geq0,\, \xi_2\in\RR.
\end{equation}
Since $\cA$, the linearization along the front $\ow$, is unitary equivalent to $\cM_{\hatt A}$, from \eqref{t5.5.4} we see that the semigroup generated by $\cA$ is bounded, proving that the front $\ow$ is Lyapunov linearly stable.
\end{proof}

\appendix

\section{Some Examples}\label{s-a}
In this section we give two examples to show that by perturbing an operator satisfying Hypotheses (H1) and (H2) we might have unstable spectrum, even if the space is finite dimensional and the perturbation is a self-adjoint, bounded, uniformly negative definite operator. 
\begin{example}\label{example-point-spectrum-jump}
Let $J_0=\begin{bmatrix}
0&1\\
0&0\end{bmatrix}$ and $Z_0=\begin{bmatrix}
1&b_0\\
b_0&\frac{1}{4}\end{bmatrix}$ with $b_0\in\big(-\frac{1}{2},\frac{1-\sqrt{2}}{2}\big)$. 
One can readily check that 
\begin{equation}\label{e1.1.1}
\Big\langle Z_0\left(\begin{matrix}u\\
v
\end{matrix}\right),\left(\begin{matrix}u\\
v
\end{matrix}\right)\Big\rangle_{\CC^2}=|u|^2+2b_0\mathrm{Re}(u\overline{v})+\frac{1}{4}|v|^2\geq(|u|-b_0|v|)^2+\Big(\frac{1}{4}-b_0^2\Big)|v|^2\;\mbox{for any}\;u,v\in\CC.
\end{equation}
From \eqref{e1.1.1} we infer that $Z_0$ is a symmetric, positive definite matrix.  Since  $b_0\in\big(-\frac{1}{2},\frac{1-\sqrt{2}}{2}\big)$, a simple computation shows that $\sigma(J_0-Z_0)=\{\mu_0^*,\lambda_0^*\}$, where $\mu_0^*=-\frac{5}{8}-\sqrt{\frac{25}{64}-(1+4b_0-4b_0^2)}<0$ and $\lambda_0^*=-\frac{5}{8}+\sqrt{\frac{25}{64}-(1+4b_0-4b_0^2)}>0$. We obtain 
\begin{equation}\label{e1.1.2}
b_1:=\frac{1}{4}\min\{\lambda_0^*,m_0\}>0,\;\mbox{where}\;m_0=\min\Big\{\Big|Z_0^{\frac{1}{2}}\left(\begin{matrix}u\\
v
\end{matrix}\right)\Big|^2:\Big|\left(\begin{matrix}u\\
v
\end{matrix}\right)\Big|=1\Big\}.
\end{equation}
It follows that $\sigma(J_0-b_1I_2)=\{-b_1\}$ is stable, $Y_0:=b_1I_2-Z_0$ is a symmetric, uniformly negative definite matrix and $\lambda_0^*\in\sigma(J_0-b_1I_2+Y_0)=\sigma(J_0-Z_0)$. We introduce the matrices
\begin{equation}\label{e1.1.3}
A_0=\begin{bmatrix}
-b_1&1&0\\
0&-b_1&0\\
0&0&0\end{bmatrix}\;\mbox{and}\; W_0=\begin{bmatrix}
b_1-1&-b_0&0\\
-b_0&b_1-\frac{1}{4}&0\\
0&0&-1\end{bmatrix}.
\end{equation}
One can readily check that $0$ is a simple eigenvalue of $A_0$ and that $\sigma(A_0)=\{0,-b_1\}$. Since any bounded operator is sectorial, we immediately infer that $A_0$ defined in \eqref{e1.1.3} satisfies Hypotheses (H1) and (H2). Since $W_0=Y_0\oplus -1$ we have $W_0$ is a symmetric, uniform negative definite matrix. However, $A_0+W_0=(J_0-Z_0)\oplus-1$, which shows that $\sigma(A_0+W_0)=(J_0-Z_0)\cup\{-1\}$, hence $\lambda_0^*\in\sigma(A_0+W_0)\cap(0,\infty)$. 
\end{example}
This example shows that by perturbing an operator that satisfies Hypotheses (H1) and (H2) we might generate  unstable  point spectrum, even if the perturbation is symmetric, negative definite. Using a similar argument, one can see that the essential spectrum can become unstable under the same type of perturbation. 
\begin{example}\label{example-essential-spectrum-jump}
We set $\bH=L^2([1,2],\CC^2)\oplus\CC$ and let $\tA_0,\tW_0:\bH\to\bH$ be the bounded linear operators defined by 
\begin{equation}\label{def-tilde-A-W-0}
\tA_0=\cM_{J_1}\oplus 0,\quad \tW_0=\cM_{Y_1}\oplus-1,
\end{equation}
where $\cM_{J_1}$ and $\cM_{Y_1}$ are the multiplication operators on $L^2([1,2],\CC^2)$ by the matrix valued functions $J_1,Y_1:[1,2]\to\CC^{2\times2}$ defined by $J_1(s)=s(J_0-b_1I_2)$ and $Y_1(s)=sY_0=s(b_1I_2-Z_0)$. Here the matrices $J_0$, $Z_0$ and $Y_0$ were introduced in Example~\ref{example-point-spectrum-jump} above. We recall that $\sigma(J_0-b_1I_2)=\{-b_1\}$ and $\sigma(J_0-Z_0)=\{\mu_0^*,\lambda_0^*\}$, with $\mu_0^*<0<\lambda_0^*$.

Since the matrix $Y_0$ is symmetric, uniform negative definite, we immediately infer that $\cM_{Y_1}$ is self-adjoint, uniform negative definite on $L^2([1,2],\CC^2)$. Hence, $\tW_0$ is self-adjoint, uniformly negative definite on $\bH$. Moreover, since the functions $J_1$ and $Y_1$ are continuous by \cite{HW}, it follows that
\begin{align}\label{e1.2.1} 
\sigma(\cM_{J_1})&=\sigma_{\mathrm{ess}}(\cM_{J_1})=\overline{\bigcup_{s\in[1,2]}\sigma(J_1(s))}=[-2b_1,-b_1],\nonumber\\
\sigma(\cM_{J_1+Y_1})&=\sigma_{\mathrm{ess}}(\cM_{J_1+Y_1})=\overline{\bigcup_{s\in[1,2]}s\sigma(J_0-Z_0)}=[-2\mu_0^*,-\mu_0^*]\,\cup[\lambda_0^*,2\lambda_0^*].
\end{align}
Since the linear operator $A_0$ is bounded on $\bH$ we have it is sectorial. Moreover, from \eqref{def-tilde-A-W-0} and \eqref{e1.2.1} we infer that $0$ is a simple eigenvalue of $\tA_0$ and $\sigma(\tA_0)\setminus\{0\}=\sigma(\cM_{J_1})=[-2b_1,-b_1]$, thus
$\sup\,\mathrm{Re}\big(\sigma(\tA_0)\setminus\{0\}\big)=-b_1<0$. Hence, $\tA_0$ satisfies Hypotheses (H1) and (H2). Even if $\tW_0$ is self-adjoint, uniform negative definite on $\bH$, from \eqref{e1.2.1} we have  $\sigma_{\mathrm{ess}}(\tA_0+\tW_0)=\sigma_{\mathrm{ess}}(\cM_{J_1+Y_1})=[-2\mu_0^*,-\mu_0^*]\,\cup[\lambda_0^*,2\lambda_0^*]$, which shows that $\sigma_{\mathrm{ess}}(\tA_0+\tW_0)\cap(0,\infty)=[\lambda_0^*,2\lambda_0^*]$ is nonempty.

\end{example}

\end{document}